\documentclass[11pt]{article}
\usepackage{graphicx}
\usepackage{epsfig}
\usepackage{amssymb, amsmath}
\usepackage{amsthm}
\usepackage{setspace}
\usepackage[top = 1in, bottom = 1in, left = 1.2in, right =
1.2in]{geometry}

\newtheorem{theorem}{Theorem}[section]

\newtheorem{claim}[theorem]{Claim}

\newtheorem{lemma}[theorem]{Lemma}

\newtheorem{proposition}[theorem]{Proposition}
\newtheorem{remark}[theorem]{Remark}

\def\D{\mathcal{D}}
\def\R{\mathbb{R}}
\def\T{\mathcal{T}}
\def\N{\mathcal{N}}
\def\G{\mathcal{G}}

\makeatletter
\def\nrfootnote{\@ifnextchar[\@xfootnote{\stepcounter\@mpfn
\protected@xdef\@thefnmark{\thempfn}%
\@footnotetext}}

\def\blfootnote{\xdef\@thefnmark{}\@footnotetext}
\makeatother 

\begin{document}

\singlespacing

\title{Well-posedness of compressible Euler equations \\ in a physical vacuum}

\author{Juhi Jang\thanks{Courant Institute. Email: \;juhijang@cims.nyu.edu \; Supported in part by NSF Grant DMS-0908007}  \; and \;  Nader Masmoudi\thanks{Courant Institute. Email: masmoudi@cims.nyu.edu Supported in part by NSF Grant DMS-0703145}}

\date{\today}

\maketitle

\begin{abstract}
An important  problem in gas and fluid dynamics is to understand the behavior of vacuum states, namely the behavior of the system in the presence of vacuum. In particular,  physical vacuum, in which the boundary moves with a nontrivial finite normal  acceleration,  naturally arises in the study of the motion of gaseous stars or shallow water. Despite its importance,  there are only  few mathematical results available near vacuum. The main difficulty lies in the fact that the physical systems become degenerate along the vacuum boundary.  In this paper, we establish the local-in-time well-posedness of  three-dimensional compressible Euler equations for polytropic gases with  physical vacuum by considering the problem as a  free boundary  problem.  
\end{abstract}




\section{Introduction}

Compressible Euler equations of isentropic, ideal gas dynamics take the form in Eulerian coordinates as follows:
\begin{equation}\label{EE}
 \begin{split}
  \partial_t\rho+\text{div} (\rho u)&=0\,,\\
\partial_t(\rho u)+\text{div} (\rho u\otimes u)+\text{grad}\,p &=0\,.
 \end{split}
\end{equation}
Here $\rho\,,\,u$ and $p$
denote respectively the
density, velocity, and pressure of gas. In considering the polytropic gases,
the constitutive relation, which is also
called the equation of state, is given by
\begin{equation}\label{gamma}
 p=K\rho^\gamma
\end{equation}
 where $K$ is an
entropy constant and $\gamma>1$ is the adiabatic gas exponent. 

Compressible Euler system \eqref{EE} is a fundamental example of a system of hyperbolic conservation laws.  The first and second equations express respectively
 conservation of mass and momentum. It is well-known \cite{Daf} that the system \eqref{EE} is strictly hyperbolic if the density is bounded below from zero: $\rho>0$.  When the initial density function contains a vacuum, the vacuum boundary $\Gamma$ is defined as
\[
 \Gamma=cl\{(t,x):\rho(t,x)>0\}\cap cl\{(t,x):
\rho(t,x)=0\}\,
\]
where $cl$ denotes the closure. It is convenient to introduce the
sound speed $c$ of the Euler equations \eqref{EE}
$$c=\sqrt{\frac{d}{d\rho}p(\rho)}\quad (=\sqrt{K\gamma}\rho^{\frac{\gamma-1}{ 2}}
\text{ for polytropic gases})$$ We recall that for one-dimensional flows, $u\pm c$ are the characteristic speeds of the system. If $\Gamma$ is nonempty,
the Euler system fails to be strictly hyperbolic
along $\Gamma$, namely degenerate hyperbolic. 
 
A vacuum boundary $\Gamma$ is called \textit{physical} if the normal acceleration near the boundary  is
bounded from below and above: 
\begin{equation}\label{pvb}
-\infty < \frac{\partial c^2}{\partial n}<0
\end{equation}
in a small neighborhood of the boundary, where
$n$ is the outward unit normal to $\Gamma$.  In other words, the pressure or the enthalpy 
$(=c^2/(\gamma-1))$ accelerates the boundary in the normal direction. This physical vacuum can be realized
by some self-similar solutions and
stationary solutions for different physical systems such as Euler
equations with damping, Euler-Poisson or  Navier-Stokes-Poisson systems for gaseous stars \cite{J0,J,LY2,Y}. 

  Recently, the physical vacuum has gotten a great deal of attention among mathematical community 
 (see \cite{IMA} and the review paper \cite{JM1}).  
Despite its physical importance, the local existence theory of
classical solutions featuring the physical vacuum boundary even for
one-dimensional flows was only  established recently. This is because
if the physical vacuum boundary condition \eqref{pvb} is assumed,
the classical theory of hyperbolic systems can not be directly applied
\cite{LY2,Y}: the characteristic speeds $u\pm c$ become singular with
infinite spatial derivatives near the vacuum boundary and this
singularity creates an analytical difficulty in standard Sobolev spaces. Local existence
for  the physical vacuum states  was given by the authors
\cite{JM}.  In \cite{JM}, the authors consider  
 the  one-dimensional Euler equations in mass Lagrangian
coordinates. Existence was proved  using  a new structure lying upon the
physical vacuum  in the framework of free boundary problems. The nonlinear energy spaces in \cite{JM} are  designed to guarantee such \textit{minimal} regularity that the physical vacuum \eqref{pvb} is realized.  Coutand and Shkoller \cite{CS09}
constructed $H^2$-type solutions with moving boundary in Lagrangian coordinates based on Hardy inequalities and  degenerate parabolic regularization. 

For multi-dimensional flows, parallel to a lot of activity and important progress in free surface boundary problems,  in particular regarding incompressible fluids \cite{ASL08,AM05,AM,CS07,GMS09,Lannes05,SZ08,Wu09,ZZ08}, more recently, there are some works trying to prove the local well-posedness of physical vacuum in 3D.    
Coutand, Lindblad and Shkoller \cite{CSL09} 
established a priori estimates based on time differentiated energy estimates and 
elliptic estimates for normal derivatives for $\gamma=2$.  As noted in \cite{CSL09}, in order to carry out the additional elliptic estimates,
sufficient smoothness of solutions was assumed and its justification by their energy 
should require additional work.  The purpose of this article is to provide a new analysis of physical vacuum based on a hyperbolic type of weighted energy estimates of tangential and normal derivatives and to establish the local well-posedness of the three-dimensional  Euler system with free surface boundary which moves with nonzero finite acceleration towards normal to the boundary. 
Independently of this work, Coutand and Shkoller 
\cite{CS10} extended their one-dimensional methodology and combined it  with the a priori estimates given in \cite{CSL09} to construct smooth solutions  of  the three-dimensional Euler equations. We were also informed that  Lindblad \cite{Lind}
has a similar result using the linearized compressible Euler equations with
a Nash-Moser iteration.  The methods are very  different.

\subsection{Existence theories of compressible Euler flows}

Before we formulate our problem, we briefly review some existence theories of compressible  flows with vacuum states from various aspects. We will not attempt to address exhaustive references in this paper.   In the absence of vacuum, namely when the system is strictly hyperbolic everywhere, one can use the theory of symmetric
hyperbolic systems developed by Friedrichs-Lax-Kato \cite{Fri,Kato,Lax} to construct smooth solutions;
for instance, see Majda \cite{Majda84}. The breakdown of classical solutions was demonstrated by Sideris \cite{Sideris85}.

When the initial datum is compactly supported, there are at least three  ways of looking
at the problem. The first consists in solving  the Euler equations
in the whole space and requiring that the system \eqref{EE} holds in the sense 
of distribution  for all $x \in \mathbb{R}^d$
and $t \in [0,T]$.  This is in particular the strategy used to 
construct global weak solutions (see for instance Diperna \cite{Diperna83} 
and \cite{Chen97,LPS96}).  The second way consists in symmetrizing the system 
first and then solving it using the theory of symmetric hyperbolic system. 
Again the symmetrized form has to be solved in the whole space. 
 The third  way is to require the Euler equations
to hold on the set $\{(t,x):\rho(t,x)>0\}$ and write an equation for $\Gamma$.
 Here, the vacuum boundary $\Gamma$ is part of the unknown: this is a free boundary problem and
in this case, an appropriate boundary condition at vacuum is necessary.
 
In the first and second  ways, 
there is no need of knowing exactly the
position of the vacuum boundary. DiPerna used the theory of compensated compactness 
to pass to the limit weakly in a parabolic approximation of the system and recovered 
a weak solution of the Euler system (see also \cite{LPS96} where a  kinetic formulation 
of the system was also used). 
  Makino, Ukai and Kawashima \cite{MUK86} wrote
the system in a symmetric hyperbolic form which allows the density
to vanish. The system they get is not equivalent to the Euler
equations when the density vanishes. This special symmetrization was also used for the
Euler-Poisson system. This  formulation was also used by Chemin \cite{Chemin90} to prove
the  local existence of regular solutions in the sense that $
c, u \in C([0,T); H^m(\mathbb{R}^d)) $ for
some $m > 1 + d/2$ and $d$ is the space dimension (see also Serre
\cite{S} and Grassin
\cite{Grassin98}, for some global existence result of classical solutions
under some special conditions on the initial data, by extracting a
dispersive effect after some invariant transformation). However, it was noted in \cite{Makino92,MU87} that
the requirement that $c$ is continuously
differentiable excludes many interesting solutions such as the
stationary solutions of the Euler-Poisson system which have a behavior of
the type  $\rho \sim |x-x_0|^{\frac{1} {\gamma -1}}$, namely $c^2\sim |x-x_0|$ near the vacuum
boundary. Indeed, Nishida in \cite{Nishida} suggested to consider a free boundary problem which includes this kind of
 singularity caused by vacuum, not shock wave singularity.

For the third  way: the free boundary problem, we divide into a few cases according to the initial behavior of the sound speed $c$. For simplicity, let the origin be the initial vacuum contact point ($x_0=0$). And let $c\sim |x|^h$.
 When $h\geq 1$, namely initial contact to vacuum is sufficiently smooth, Liu and Yang
\cite{LY1} constructed the local-in-time solutions to one-dimensional Euler equations with
damping by using the energy method  based on the
adaptation of the theory of symmetric hyperbolic system. They also prove that $c^2$ can not be smooth across
$\Gamma$ after a finite time. We note that in these regimes there is no acceleration along the vacuum
boundary. For $0<h<1$, the initial contact to vacuum is only Holder continuous. In particular,
the corresponding behavior to $h=1/2$ is the case of physical vacuum \cite{CS09,CS10,CSL09,JM}.  For $0<h<1/2$ and $1/2<h<1$, its boundary behavior is believed  to be ill-posed;  
indeed, we conjecture that it should instantaneously change into the physical vacuum. However,
there is no mathematical justification available so far. 

The case $h=0$ is when there is no continuous initial contact of the density with vacuum.
It can be considered as either
Cauchy problem or free boundary problem. An example of Cauchy problem when $h=0$ is the 
  Riemann problem for
genuinely discontinuous initial datum (see for instance  \cite{Bouchut04,GB}).
An example of a free boundary problem when
$h=0$ is the work by Lindblad \cite{L1} where the density is positive at the vacuum boundary.  

Having the local existence theory of vacuum states, the next
important question is whether
such a local solution exists globally in
time or how it breaks down.
The study of vacuum free boundary
 automatically excludes the breakdown of solutions caused by vacuum,
which is one possible scenario of the breakdown for
positive solutions to compressible Euler equations \eqref{EE}.
It was shown in
\cite{LS} that the shock waves vanish at the vacuum and the singular
behavior is similar to the behavior of the centered rarefaction
waves corresponding to the case when $c$ is regular \cite{LY2}, which indicates that vacuum has a regularizing
effect. Therefore it would be very interesting to investigate
the long time behavior of vacuum states.  

When there is damping, based on self-similar behavior, Liu conjectured \cite{L2} that
time asymptotically, solutions to
Euler equations with damping should behave like the ones
to the porus media equation, where the canonical boundary
is characterized by the physical vacuum condition
\eqref{pvb}.  This conjecture was established by Huang, Marcati and Pan \cite{HMP} in the
 framework of the entropy solution where the method of compensated
compactness yields a global weak solution in $L^\infty$. But in their work, there is no way of
tracking the vacuum boundary. It would be interesting to investigate the asymptotic relationship between smooth solutions obtained by solving a free boundary problem of Euler equations with damping and smooth solutions of the porus media equation. 

\subsection{Other interesting vacuum states}\label{other}

Physical vacuum as well as other vacuum states appear in the theory of other physical systems,   where we believe the  methodology of this article could be applied. Here we briefly present some of them; see \cite{JM1} for more detailed discussion.  

The study of vacuum is  important in understanding viscous flows \cite{Fei,Lions}.
 When vacuum appears initially, studying Cauchy problems
for compressible Navier-Stokes equations with constant viscosity
coefficients yields somewhat negative results: for instance, a
finite time blow-up for nontrivial compactly supported initial
density \cite{Xin} and a failure of continuous dependence on initial
data \cite{HS}. There are some existence theories available with the
physical vacuum boundary for one-dimensional Navier-Stokes free boundary
problems \cite{LXY}, for three-dimensional Navier-Stokes-Poisson equations with radial symmetry \cite{J}, and for other related models  \cite{dz,mom}. On the other hand, to resolve the issue of no
continuous dependence on initial data in \cite{HS} for constant
viscosity coefficient, a density-dependent viscosity coefficient was
introduced in \cite{LXZ}. Since then, there has been a lot of
studies on global weak solutions for various models and
stabilization results under gravitation and external forces: see
\cite{LLX,ZF2} and the references therein. Despite 
significant progress over the years, many  interesting and important
questions are still  unanswered especially for general
multi-dimensional flows.

The relativistic Euler equations are known to be symmetric hyperbolic
away from vacuum \cite{MU95I}. A particular interest is
compactly supported relativistic flows which for instance can be
applied to the dynamics of stars in the context of special
relativity. Whether one can
extend the theory of free boundary problems including physical
vacuum developed for non-relativistic Euler equations to
relativistic case is an open problem. It turns out that vacuum states also arise in the theory of magnetohydrodynamics (MHD), another interesting system of hyperbolic conservation laws arising
from electromechanical phenomena \cite{Daf,Grad}. 
 Due to the interplay
between the scalar pressure of the fluid and the anisotropic
magnetic stress, vacuum states are richer than in hydrodynamics and
their rigorous study in the context of nonlinear partial
differential equations seems to be widely open.

Lastly, it is interesting to point out the connection between vacuum states of degenerate hyperbolic systems and degenerate elliptic and parabolic equations. One of the  main difficulty of  studying the Euler system with a 
 free boundary  is that 
it leads to a degenerate hyperbolic equation due to the fact that 
the density vanishes at the free boundary. 
For this we need techniques coming from degenerate elliptic equations; for instance, see \cite{Baouendi67}. Also, similar problems arise 
in degenerate parabolic equations  for instance, in  porous medium  equations \cite{CF}, 
in  thin film equations   \cite{GHO08}, 
in the study of polymeric flows \cite{Masmoudi08cpam}.

In the next section, we formulate the problem, introduce notations and state the main result.

\section{Lagrangian Formulation and Main Result}\label{2}

\subsection{Derivation of the system in Lagrangian  coordinates}
 
The boundary moves  with a finite normal acceleration under the 
  physical vacuum condition \eqref{pvb} and it is part of the unknown.  
The vacuum free boundary problem is studied in Lagrangian coordinates where the free boundary is fixed. 

For smooth solutions, the Euler equations \eqref{EE} can be written in Eulerian coordinates 
as :
\begin{equation}
 \begin{split}\label{EEE}
  (\partial_t+u\!\cdot\!\nabla)\rho +\rho(\nabla\!\cdot\!u)=0\,,\\
\rho (\partial_t+u\!\cdot\!\nabla)u + K\nabla\rho^\gamma =0\,.
 \end{split}
\end{equation}
Let $\eta(t,x)$ be the position of the gas particle $x$ at time $t$ so that
\begin{equation}\label{flow}
\eta_t = u(t, \eta(t,x))\text{ for }t>0 \text{ and }\eta(0,x)=x\, \text{ in } \Omega\,.
\end{equation}
As in \cite{CSL09}, we define the following Lagrangian quantities:
\[
\begin{split}
 v(t,x)&\equiv u(t, \eta(t,x))\;\text{(Lagrangian velocity)}\\ f(t,x)&\equiv \rho(t,\eta(t,x))\;
 \text{(Lagrangian density)}\\
A&\equiv[D\eta]^{-1}\;\text{(inverse of deformation tensor)}\\
J&\equiv\det D\eta\;\text{(Jacobian determinant)}\\ a&\equiv
JA\;\text{(transpose of cofactor matrix)}
\end{split}
\]
We use Einstein's summation convention and the notation $F,_k$ to denote the $k^{\text{th}}$-partial derivative of $F$: $\partial_kF$.  Both expressions will be used throughout the paper. We use $i,j,k,l,r,s$ to denote 1, 2, 3. 
The Euler equations \eqref{EEE} read as follows:
\begin{equation}\label{Lag}
\begin{split}
 f_t+fA^j_iv^i,_{j}=0\,,\\ fv_t^i+ KA_i^k f^\gamma,_{k}=0\,.
\end{split}
\end{equation}
Since
$$J_t=JA_i^jv^i,_j \text{ and }J(0)=1\,,$$
together with the equation for $f$, we find that $$fJ=\rho_0$$ where
$\rho_0$ is given initial density function. Thus, using
$A^k_i=J^{-1}a_i^k$, \eqref{Lag} reduce to the following:
\begin{equation}\label{Lag2}
\rho_0v^i_t +Ka_i^k({\rho_0}^\gamma{J}^{-\gamma}),_k=0
\end{equation}
along with
\begin{equation}\label{eta}
 \eta_t^i =v^i \,.
\end{equation}

Now let $w$ be 
\begin{equation}\label{w}
w\equiv K\rho_0^{\gamma-1}\,.
\end{equation}
Note that $\tfrac{\gamma}{\gamma-1}w$ is the initial enthalpy. 
We are interested in   smooth initial enthalpy profiles satisfying the physical vacuum condition \eqref{pvb}: 
\begin{equation}\label{wcondition}
\begin{split}
w=0 \text{ on } \partial\Omega\,(=\Gamma(0))\,, \; w>0 \text{ in } \Omega\,,\\
\frac{1}{C} d(x)\leq  w(x) \leq C d(x)\text{ in } \Omega\,,
\end{split}
\end{equation}
where $d(x)$ is the distance function to the boundary. 
The equation \eqref{Lag2} now takes the form
\begin{equation}\label{general}
w^\alpha v^i_t +(w^{1+\alpha}\,A_i^k{J}^{-1/\alpha}),_k=0\,,
\end{equation}
where $$\alpha\equiv\frac{1}{\gamma-1}\;.$$ We have used the Piola
identity \eqref{Piola} to get \eqref{general} from \eqref{Lag2}. Note that $\alpha>0$ and
$\alpha\rightarrow\infty$ as $\gamma\rightarrow 1$. Since
$\eta^i_t=v^i$, the equation \eqref{general} reads as an $\eta$ equation:
\begin{equation}\label{deacoustic}
w^\alpha \eta^i_{tt}
+(w^{1+\alpha}\,A_i^k{J}^{-1/\alpha}),_k=0\,.
\end{equation}
Thus the equation \eqref{deacoustic} can be viewed as a degenerate nonlinear acoustic
(wave) equation for $\eta$. Multiply \eqref{deacoustic} by
$\eta_t^i$ and integrate over $\Omega$: the zeroth order energy estimates formally lead to the
following energy conservation
\begin{equation}\label{energy}
\frac{d}{dt}\int_\Omega \left\{\frac{1}{2}w^\alpha |v|^2+\alpha
w^{1+\alpha} J^{-1/\alpha}\right\} dx =0 \text{ denoted by
}\frac{dE}{dt}=0\,.
\end{equation}
Note that due to the vanishing factor $w^{1+\alpha}$ at the 
free boundary, all the boundary terms from the
integration by parts disappeared. This energy conservation is
equivalent to the more familiar form in Eulerian coordinates:
\[
 \frac{d}{dt}\int_{\Omega(t)}\left\{\frac12\rho |u|^2 +\frac{p}{\gamma-1} \right\}dx =0
\]
which is the conservation of the physical energy. 

\begin{remark}
The degenerate  nonlinear acoustic system  \eqref{deacoustic} is in a sense equivalent 
to the Euler equation \eqref{EEE} with physical vacuum. Notice that the initial datum  
$w = K \rho_0^{\gamma-1}$ is a parameter in the equation and that one recovers the 
Lagrangian density by taking $f  = {\rho_0}{J}^{-1}$. 
\end{remark}

\subsection{Differentiation of $A$ and $J$}

In order to have
sufficient regularity so that the flow map $\eta$ is guaranteed to
be non-degenerate and smooth and in particular the Jacobian determinant $J$ is to be bounded
away from zero and smooth, we need estimates of sufficiently high
order derivatives of $\eta$ or $v$.
The minimal number of derivatives needed will be determined according to
 the strength of degeneracy, namely vanishing exponent $\alpha$. Here we present the differentiation of $A$ and $J$.

Differentiating the inverse of deformation tensor,  since $A\cdot
[D\eta] =I$, one obtains  
\begin{equation}\label{DA}
\partial_t A^k_i=- A^k_rv^r,_sA^s_i\;;\quad \partial_lA^k_i=-
A^k_r\partial_l\eta^r,_sA^s_i
\end{equation}
Differentiating the Jacobian determinant, one obtains 
\begin{equation}\label{DJ}
\partial_tJ = JA^s_r v^r,_s\;; \quad \partial_l J= JA^s_r \partial_l\eta^r,_{s}
\end{equation}
For the cofactor matrix $a=JA$, from \eqref{DA} and \eqref{DJ},  one obtains the following Piola identity. 
\begin{equation}\label{Piola}
a^k_i,_k=0
\end{equation}

\subsection{Notation}

 For a given vector field $F$ on $\Omega$, we use $DF$, $\text{div}F$, $\text{curl}F$ to denote its full gradient, its divergence, and its curl: 
\[
\begin{split}
[DF]^i_j & \equiv F^i,_j \\ 
\text{div}F \;&\equiv  F^r,_r\\
[\text{curl}F]^i & \equiv \epsilon_{ijk} F^k,_j
\end{split}
\] 
Here $\epsilon_{ijk}$ is the Levi-Civita symbol: it is 1 if $(i,j,k)$ is an even permutation of $(1,2,3)$, -1 if $(i,j,k)$ is an odd permutation of $(1,2,3)$, and 0 if any index is repeated. 

We introduce the following Lie derivatives along the flow map $\eta$:
\[
\begin{split}
 [D_\eta F]^i_r&\equiv A^s_rF^i,_s\\ \text{div}_\eta F\;&\equiv A^s_rF^r,_s \\
[\text{curl}_\eta F]^i &\equiv \epsilon_{ijk}A^s_jF^k,_s
\end{split}
\]
which indeed correspond to Eulerian full gradient, Eulerian divergence, and Eulerian curl written in Lagrangian coordinates. 
When $\eta$ is the identity map -- for instance, the initial state of the flow map is the identity map  as in \eqref{flow} -- these 
Lie derivatives are the standard full gradient, divergence, and curl.   In addition, it is convenient to introduce the anti-symmetric  curl matrix $\text{Curl}_\eta F$ 
\[
[\text{Curl}_\eta F]^i_j \equiv  A^s_jF^i,_s - A^s_iF^j,_s 
\]
Note that $\text{Curl}_\eta F$ is a matrix version of a vector $\text{curl}_\eta F$ and that  $|\text{Curl}_\eta F|^2 
=2 |\text{curl}_\eta F|^2$ holds. We will use both $\text{curl}_\eta $ and $\text{Curl}_\eta $.

\subsection{Energy and main result}\label{energy-main}

\def\o{\overline}

To state our result, we need to introduce some  energy  which behaves differently 
for normal and tangential derivatives close to the boundary. We first introduce the tangent vector fields and the normal vector field for the initial domain. 

For given smooth initial domain $\Omega$,   
 we consider a smooth function 
$\phi : \R^3 \to \R$  such that $ \{ x, \phi(x) = 0 \} = \partial \Omega   $ and  
$ \{ x, \phi(x)  >  0 \} =  \Omega   $, and that  $ |\nabla \phi (x)  | = 1    $
 for all $x \in \Omega_\delta = \{ x \in \Omega, \, d(x) < \delta  \} $ for some small $\delta > 0$. In particular 
$ \phi(x) = d(x)  $ in $\Omega_\delta $. Let $\T$ be a finite  set of   (smooth enough) vector fields in 
$\Omega$ which are   tangent to the boundary and $\N$ the set formed by the 
 vector field $\zeta$  which 
is  normal to the boundary and such that $\zeta = \nabla \phi $ for $x \in \Omega_\delta.  $
We also assume that at each point $x$ in $\Omega$ the vectors of $\T $ and $\N$ span 
the whole space $\R^3$ and that the second order differential operator 
  $\sum_{\beta \in \T} \beta^* \beta + \zeta^* \zeta $ is strictly 
elliptic.  In the simplified case where $\Omega  =\mathbb{T}^2\times
(0,1) $, we can just take $  \partial_\zeta = \partial_3  $   and 
 $\T =\{ \partial_1, \partial_2 \}  $.

We define the  energies $ \o{ \mathcal{B} } ^{N},   $   $ \o{  \mathcal{C} } ^{N}, $  
 $   \o{   \mathcal{D}  }^{N} $
and  $\o { \mathcal{E}}^{N} $  by 
 \begin{equation}\label{bar-ener}
\begin{split}
\o{\mathcal{B}}^{N}(v)\;
 &\equiv \frac12\int_\Omega   \sum_{|m|+n=0}^N      d^{1+\alpha+n}
|\text{curl}_\eta \partial_\beta^m\partial_\zeta^n v|^2dx \\
\o{\mathcal{C}}^{N}(\eta)\;
 &\equiv \frac12\int_\Omega  \sum_{|m|+n=0}^N     d^{1+\alpha+n} 
| \text{curl}_\eta \, \partial_\beta^m\partial_\zeta^n \eta|^2dx \\
\o{\mathcal{D}}^{N}(\eta)\;&\equiv
\frac{1}{2\alpha} \int_\Omega   \sum_{|m|+n=0}^N     d^{1+\alpha+n}
|\text{div}_\eta\partial_\beta^m\partial_\zeta^n \eta|^2dx\\
\o{\mathcal{E}}^{N}(\eta,v)&\equiv
\frac12\int_\Omega      \sum_{|m|+n=0}^N     d^{\alpha+n} |\partial_\beta^m\partial_\zeta^n
v|^2  +  d^{1+\alpha+n} 
|D_\eta\partial_\beta^m\partial_\zeta^n \eta|^2dx
\end{split}
\end{equation}
In the above formulae, the sum is made over $m \in \mathbb{N}^{I}$ and 
$\partial_\beta^m$ denotes the derivative  
 $\partial_{\beta_1}^{m_1} \partial_{\beta_2}^{m_2}...\partial_{\beta_{I}}^{m_{I}}    $ where $I = \hbox{Card} \T $. 
The total energy is defined by
\begin{equation}\label{energyf-general}
\begin{split}
 \o{\mathcal{TE}}^{N}\equiv \o{\mathcal{TE} } ^N(\eta,v)\equiv 
 \o{\mathcal{E}}^{N}(\eta,v)+ \o{\mathcal{B}}^N(v)
\end{split}
\end{equation}
Since $\eta(0,x)=x$, the total energy for initial data $v(0,x)=u_0$ is given by 
\begin{equation}\label{ie-general}
  \o{\mathcal{TE}}^{N}(0)=\frac32\int_\Omega d^{1+\alpha}dx+ \sum_{|m|+n=0}^N\frac{1}{2}\int_\Omega d^{\alpha+n}
 |\partial_\beta^m\partial_\zeta^n u_0|^2 +d^{1+\alpha+n}
 |\text{curl}\,\partial_\beta^m\partial_\zeta^n u_0|^2dx 
\end{equation}

We also introduce the  function spaces $\o{X}^{\alpha,b}$, $\o{Y}^{\alpha,b}$
\begin{equation}\label{XY-general}
\begin{split}
\o{X}^{\alpha,b}&\equiv \{d^{\frac{\alpha}{2}}F\in L^2(\Omega) : \int_{\Omega}d^{\alpha+n}|\partial_\beta^m\partial_\zeta^n F|^2 dx<\infty\,,\,0\leq |m|+n\leq b\}\\
\o{Y}^{\alpha,b}&\equiv \{d^{\frac{1+\alpha}{2}}D F\in L^2(\Omega) : \int_{\Omega}d^{1+\alpha+n} 
 |D  \partial_\beta^m\partial_\zeta^n F|^2 dx<\infty\,,\,0\leq |m|+n\leq b\}
\end{split}
\end{equation}
equipped with the following norms: 
\[
\begin{split}
\|F\|^2_{\o{X}^{\alpha, b}}\equiv \sum_{|m|+n=0}^b  \int_{\Omega}d^{\alpha+n}|\partial_\beta^m\partial_\zeta^n F|^2 dx\;;\;\;\;\;
\|F\|^2_{\o{Y}^{\alpha, b}}\equiv \sum_{|m|+n=0}^b\int_{\Omega}d^{1+\alpha+n}  
 |D  \partial_\beta^m\partial_\zeta^n F|^2 dx
\end{split}
\]
Note that 
\[
\o{\mathcal{E}}^N\sim \|v\|_{\o{X}^{\alpha,N}}^2+\|\eta \|_{\o{Y}^{\alpha,N}}^2\,.
\]

We are now ready to state the main result of this article. 

\begin{theorem}\label{thm-general} Let $\alpha>0$ be fixed and $N\geq 2[\alpha]+9$ be given. Suppose  that $Dw\in \o{X}^{\alpha,N}$ and that the initial energy \eqref{ie-general} is bounded: 
$\o{\mathcal{TE}}^N(0)<\infty$.   Then there exist a time $T>0$ 
depending only on $\o{\mathcal{TE}}^N(0)$ and  $\| Dw \|_{\o{X}^{\alpha,N}} $
 a unique solution $(v,\eta)\in C([0,T]; \o{X}^{\alpha,N}\times \o{Y}^{\alpha,N})$ to the Euler equation 
\eqref{eta} and \eqref{deacoustic} on the time interval $[0,T]$ satisfying 
\[
\o{\mathcal{TE}}^N(\eta,v)\leq 2\o{\mathcal{TE}}^N(0)\;\text{ and }\;  \|A-I\|_\infty\leq 1/8\,.
\] 
In particular, ${2}/{3}\leq J\leq 2$. 
\end{theorem}

 Theorem \ref{thm-general} indicates that the  minimal number of derivatives needed to capture the physical vacuum \eqref{pvb} depends on the value of the adiabatic exponent $\gamma$. The smaller $\gamma$ is, the more derivatives are required to overcome stronger degeneracy caused by physical vacuum.  Indeed, this phenomenon was captured in our one-dimensional result \cite{JM} where the analysis is carried out in mass Lagrangian coordinates. While there is some similarity between our new analysis and the previous one-dimensional analysis, besides boundary geometry, there is another critical difference regarding the energy functionals: In one-dimensional $V,V^\ast$ framework, the energy space is very nonlinear in that it is not clear at all to deduce some equivalence to the standard weighted Sobolev spaces and moreover, the number of $V,V^\ast$ to define the energy space is rigid. On the other hand, in the current analysis, the energy functionals given in \eqref{energyf} are indeed equivalent to the standard linear weighted Sobolev spaces as discussed in Section \ref{ws} and also the higher regularity can be readily established. 

The method of the proof is based on a hyperbolic type of new energy estimates which consist in the instant energy estimates and the curl estimates. As soon as we linearize the Euler system, we start to see geometric structures: the full gradient, divergence and curl of flow map $\eta$ in the equations.  The new key is to extract \textit{right algebraic weighted structure} for the linearization in the normal direction such that we can directly estimate normal derivatives via the energy estimates: each time we take normal derivative, we obtain more singular (degenerate) weight while the main structure of the equation remains the same for tangential derivatives.  The effect of taking one normal derivative is worth gaining a half derivative. The new energy estimates provide a unified, systematic way of treating all the spatial derivatives. In the instant energy estimates, the curl part comes with undesirable negative sign. This can be absorbed by adding the curl estimates which will be obtained separately from the curl equation. For the construction of solutions, we implement the linear approximate schemes for  
$G=[  \sum_{\beta \in \T}  (\partial_\beta)^* \partial_\beta  - w^{-\alpha}\partial_\zeta w^{1+\alpha}\partial_\zeta+ \lambda] \eta$  with $\lambda$ big enough 
  and we use  a relaxed curl of $G$.  To build the well-posedness of linear approximate systems, we employ the duality argument as done in \cite{JM}. Finally, $\eta$ is found by solving the above degenerate elliptic equation.

In the sequel we will give a detailed proof in the simplified case 
$\Omega = \mathbb{T}^2\times
(0,1)$ and then indicate the changes to be done in the general case. 
The rest of the paper is organized as follows.  In section \ref{new}, we  
rewrite the energies in the simplified case  $\Omega = \mathbb{T}^2\times
(0,1)$ and prove some Hardy type inequalities. 
 In Section \ref{3}, we establish the a priori energy estimates. In Section \ref{exist-proof}, an approximate scheme is provided and Theorem \ref{thm-general} is proven.  In Section \ref{genera-domain}, the general domain case is discussed. In Section \ref{discussion}, we conclude the article  with a few remarks.

\section{A simplified domain} \label{new}

For simplicity of the presentation, we will present the full proof in the case when 
  the initial domain is taken as $$\Omega=\mathbb{T}^2\times
(0,1)$$ 
where $\mathbb{T}^2$ is a two-dimensional period box in $x_1,x_2$. We will then present the 
main changes to be done in the general case. 
  The initial boundary is given as $$\Gamma(0)=\{x_3=0\}\cup\{x_3=1\} \text{ as the reference vacuum boundary.}$$
  The moving vacuum
boundary is given by $\Gamma(t)=\eta(t)(\Gamma(0))$.

We use Latin letters  $i,j,k, ...$ to denote  1, 2, 3
     and that we 
  use Greek letters $\beta,\kappa,\sigma$ to denote 1, 2 only.
Recall the weight $w$ -- initial enthalpy \eqref{w} satisfying \eqref{wcondition}. We now rewrite the  various energy functionals. We use $\partial_\beta^m$ to denote $\partial_1^{m_1}\partial_2^{m_2}$ and $|m|$ to denote $|m|=m_1+m_2$.  
 
We also introduce  $\eta$ dependent energies. These energies are equivalent to the 
ones   given in \eqref{bar-ener} with the difference that $d$ is replaced by $w$ and 
that a factor $J^{-1/\alpha}$ is added. It turns out that these versions are more 
adapted to the energy estimates and have better cancellation properties. 
 
 The curl energies are defined by
\begin{equation}\label{curlE}
\begin{split}
 \mathcal{B}^N(v)\equiv  \sum_{|m|+n=1}^N\frac12\int_\Omega  w^{1+\alpha+n} J^{-1/\alpha}
|\text{curl}_\eta\partial_\beta^m\partial_3^n v|^2dx\equiv 
 \sum_{|m|+n=1}^N\mathcal{B}^{m,n}\\
 \mathcal{C}^N(\eta)\equiv  \sum_{|m|+n=1}^N \frac12\int_\Omega w^{1+\alpha+n} J^{-1/\alpha}
|\text{curl}_\eta\partial_\beta^m\partial_3^n \eta|^2dx \equiv 
 \sum_{|m|+n=1}^N\mathcal{C}^{m,n} 
\end{split}
\end{equation}
The divergence energy is defined by
\begin{equation}\label{divE}
\mathcal{D}^N(\eta)\equiv \sum_{|m|+n=1}^N \frac{1}{2\alpha} \int_\Omega w^{1+\alpha+n} J^{-1/\alpha}
|\text{div}_\eta\partial_\beta^m\partial_3^n \eta|^2dx \equiv \sum_{|m|+n=1}^N\mathcal{D}^{m,n}
\end{equation}
The instant energy is defined by 
\begin{equation}\label{energyif}
\begin{split}
 \mathcal{E}^{N}(\eta,v)& \equiv E+ \sum_{|m|+n=1}^N\frac12\int_\Omega w^{\alpha+n} |\partial_\beta^m\partial_3^n
v|^2dx +\frac12
 \int_\Omega w^{1+\alpha+n} J^{-1/\alpha}
|D_\eta\partial_\beta^m\partial_3^n \eta|^2dx \\ & \equiv  
E+ \sum_{|m|+n=1}^N\mathcal{E}^{m,n}
\end{split}
\end{equation}
The total energy is defined by
\begin{equation}\label{energyf}
\begin{split}
 \mathcal{TE}^{N}\equiv\mathcal{TE}^N(\eta,v)\equiv \mathcal{E}^{N}(\eta,v)+\mathcal{B}^N(v)
\end{split}
\end{equation}
Since $\eta(0,x)=x$, the total energy for initial data $v(0,x)=u_0$ is given by 
\begin{equation}\label{ie}
 \mathcal{TE}^{N}(0)=\mathcal{TE}^N(x,u_0)=E+\sum_{|m|+n=1}^N\frac{1}{2}\int_\Omega w^{\alpha+n}
 |\partial_\beta^m\partial_3^n u_0|^2 +w^{1+\alpha+n}
 |\text{curl}\,\partial_\beta^m\partial_3^n u_0|^2dx 
\end{equation}

Next we introduce the function spaces $X^{\alpha,b}$, $Y^{\alpha,b}$, $Z^{\alpha,b}$, which will be used in the construction of solutions to approximate scheme in Section \ref{IS}, 
associated  to our various energy functionals \eqref{curlE} -- \eqref{energyif}: 
\begin{equation}\label{XY}
\begin{split}
X^{\alpha,b}&\equiv \{w^{\frac{\alpha}{2}}F\in L^2(\Omega) : \int_{\Omega}w^{\alpha+n}|\partial_\beta^m\partial_3^n F|^2 dx<\infty\,,\,0\leq |m|+n\leq b\}\\
Y^{\alpha,b}&\equiv \{w^{\frac{1+\alpha}{2}}D_\eta F\in L^2(\Omega) : \int_{\Omega}w^{1+\alpha+n}J^{-\frac{1}{\alpha}}|D_\eta \partial_\beta^m\partial_3^n F|^2 dx<\infty\,,\,0\leq |m|+n\leq b\}\\
Z^{\alpha,b}&\equiv \{w^{\frac{1+\alpha}{2}} F\in L^2(\Omega) : \int_{\Omega}w^{1+\alpha+n}J^{-\frac{1}{\alpha}}| \partial_\beta^m\partial_3^n F|^2 dx<\infty\,,\,0\leq |m|+n\leq b\}
\end{split}
\end{equation}

\subsection{Imbedding of a weighted Sobolev space}\label{ws}

For any given $\alpha>0$ and given nonnegative integer $b$, we define the weighted Sobolev spaces $H^{\alpha,b}(\Omega)$ by 
\[
H^{\alpha,b}(\Omega)\equiv \{d^{\frac{\alpha}{2}}F\in L^2(\Omega) : \int_{\Omega}d^{\alpha}|D^k F|^2 dx<\infty\,,\,0\leq k\leq b\}
\]
with the norm 
\[
\|F\|^2_{H^{\alpha, b}}\equiv \sum_{k=0}^b  \int_{\Omega}d^{\alpha}|D^k F|^2 dx
\]
We denote the standard Sobolev spaces by $H^s$.  Then for $b\geq \alpha/2$, the weighted spaces $H^{\alpha,b}$ satisfy the following  Hary type 
 embedding \cite{KMP07}: 
\[
H^{\alpha,b}(\Omega)\hookrightarrow H^{b-\frac{\alpha}{2}}
\]
with 
\[
\|F\|_{H^{b-\alpha/2}}\;\precsim \;\|F\|_{H^{\alpha,b}}
\]

 As an application of the above embedding of weighted Sobolev spaces, we first obtain the 
embedding of $X^{\alpha,b}$ into Sobolev spaces for  sufficiently smooth $w$. 

\begin{lemma}\label{emb} For $b\geq \lceil\alpha\rceil$, 
\begin{equation*}
\|F\|_{H^{\frac{b-\alpha}{2}}}\precsim \|F\|_{X^{\alpha,b}}
\end{equation*}
In particular, for $b\geq [\alpha]+4$, 
\begin{equation*}
\|F\|_\infty \precsim  \|F\|_{X^{\alpha,b}}
\end{equation*}
\end{lemma}

If $A=[D\eta]^{-1}$ is close to the identity -- for instance, if $\|A-I\|_\infty \leq 1/8$ which implies  $2/3\leq J\leq 2$ -- then $Y^{\alpha,b}$ and $Z^{\alpha,b}$ are isomorphic to standard linear weighted Sobolev spaces like $X^{\alpha,b}$. Therefore, Lemma \ref{emb} dictates the similar embeddings for $Y^{\alpha,b}$ and $Z^{\alpha,b}$: for $b\geq \lceil\alpha\rceil+1$,  
\begin{equation*}
\|DF\|_{{H}^{\frac{b-\alpha-1}{2}}}\precsim \|F\|_{Y^{\alpha,b}}\;\;\text{ and  }\;\;
\|F\|_{{H}^{\frac{b-\alpha-1}{2}}}\precsim \|F\|_{Z^{\alpha,b}}
\end{equation*}

\section{A Priori Energy Estimates}\label{3}

In this section, we will prove the following a priori estimates.

\begin{proposition}\label{prop} Let $\alpha>0$ be given and let $N\geq 2[\alpha]+9$ and $Dw\in X^{\alpha,N}$. Suppose $\eta$ and $v$ solve \eqref{eta} and \eqref{deacoustic} for $t\in[0,T]$ 
 with $\mathcal{TE}^N(\eta,v)<\infty$ and $1/C_0\leq J\leq C_0$ for some $C_0\geq 1$. We further assume that $\eta$
and $v$ enjoy the a priori bound: for any $s=1,2,\text{ and }3$, 
\begin{equation}\label{Assumption}
 \sum_{|p|+q=0}^{[N/2]} |w^{q/2}\partial_\beta^p\partial_3^q\eta^r,_s|
+ \sum_{|p|+q=0}^{[N/2]-1}|w^{q/2}\partial_\beta^p\partial_3^qv^r,_s|<\infty
\end{equation}
 Then we obtain the following a priori estimates:
 \begin{equation}\label{e}
 \begin{split}
 \frac{d}{dt}\mathcal{E}^N(\eta,v)\leq \mathcal{F}(\mathcal{E}^N(\eta,v),\mathcal{B}^N(u_0),C_0)
\end{split}
\end{equation}
where $\mathcal{F}(\mathcal{E}^N(\eta,v),\mathcal{B}^N(u_0),C_0)$ is a continuous
function of $\mathcal{E}^N(\eta,v),\;\mathcal{B}^N(u_0),\text{ and }C_0$. In addition, we have the 
curl energy $\mathcal{B}^N$ bounded 
\begin{equation}\label{b}
 \mathcal{B}^N(v)\leq \mathcal{B}^N(u_0)+\mathcal{G}(\mathcal{E}^N(\eta,v),C_0,T)
\end{equation}
where $\mathcal{G}(\mathcal{E}^N(\eta,v),C_0,T)$ is a continuous function of $\mathcal{E}^N(\eta,v),\;C_0\text{ and }T$.  Moreover, the a priori assumption \eqref{Assumption} can be justified. 
\end{proposition}

The proof of Proposition \ref{prop} is based on the following two key
lemmas.

\begin{lemma}\label{lemma1} Assume as in Proposition \ref{prop}. Then we obtain the following 
\begin{equation}\label{e1}
  \frac{d}{dt}\left\{\mathcal{E}^N(\eta,v)+\mathcal{D}^N(\eta)
-\mathcal{C}^N(\eta)\right\}\leq
\mathcal{F}_1(\mathcal{E}^N(\eta,v),C_0)\,.
\end{equation}
\end{lemma}

\begin{lemma}\label{lemma2} Assume as in Proposition \ref{prop}. Then we obtain the following  
\begin{equation}\label{e2}
  \frac{d}{dt}\mathcal{C}^N(\eta)\leq \mathcal{F}_2(\mathcal{E}^N(\eta,v),\mathcal{B}^N(u_0),C_0)\,.
\end{equation}
Moreover, \eqref{b} is satisfied. 
\end{lemma}

By adding \eqref{e1} and \eqref{e2}, the main estimate \eqref{e}
follows. Moreover, the a priori assumption \eqref{Assumption} can be verified within $\mathcal{E}^N$ by Lemma \ref{emb} and the standard continuity argument.

As a preparation of higher order energy estimates, we first compute
the spatial derivative of $A_i^k{J}^{-1/\alpha}$.
\[
 \begin{split}
 \partial_l(A_i^k{J}^{-1/\alpha})&= {J}^{-1/\alpha}\partial_lA_i^k -\tfrac{1}{\alpha}
{J}^{-(1+\alpha)/\alpha} A_i^k \partial_lJ\\
&=-{J}^{-1/\alpha}A^k_rA^s_i\partial_l\eta^r,_s-\tfrac{1}{\alpha}{J}^{-1/\alpha}A^k_iA^s_r\partial_l
\eta^r,_s\\
&= -{J}^{-1/\alpha}A^k_rA^s_r\partial_l\eta^i,_s
-\tfrac{1}{\alpha}{J}^{-1/\alpha}A^k_iA^s_r\partial_l\eta^r,_s\\
&\quad\!\,-{J}^{-1/\alpha}A^k_r[A^s_i\partial_l\eta^r,_s-A^s_r\partial_l\eta^i,_s]
 \end{split}
\]
Thus 
\begin{equation}
\label{structure1}
 \partial_l(A_i^k{J}^{-1/\alpha})=-{J}^{-1/\alpha}A^k_r\, [D_\eta\partial_l\eta]^i_r
-\tfrac{1}{\alpha}{J}^{-1/\alpha}A^k_i\,\text{div}_\eta\partial_l\eta-{J}^{-1/\alpha}A^k_r\,
[\text{Curl}_\eta\partial_l\eta]^r_i
\end{equation}
Similarly, the time derivative of $A_i^k{J}^{-1/\alpha}$ is
given by
\begin{equation}
\begin{split}\label{structure2}
\partial_t(A_i^k{J}^{-1/\alpha})=-{J}^{-1/\alpha}A^k_r\, [D_\eta v]^i_r
-\tfrac{1}{\alpha}{J}^{-1/\alpha}A^k_i\,\text{div}_\eta v-{J}^{-1/\alpha}A^k_r\,
[\text{Curl}_\eta v]^r_i
\end{split}
\end{equation}
We remark that {after taking derivatives}, namely \textit{after linearization},  we start to see
 structures: in \eqref{structure1} and \eqref{structure2}, the
first term corresponds to the full gradient, the second to the
divergence, the last term to the curl. 
For $n\geq1$, we write
$\partial_l^n[A_i^k{J}^{-1/\alpha}]$ as follows.
\begin{equation}
\begin{split}\label{highD}
\partial_l^n[A_i^k{J}^{-1/\alpha}]
&= -{J}^{-1/\alpha}A^k_r\, [D_\eta\partial_l^n\eta]^i_r
-\tfrac{1}{\alpha}{J}^{-1/\alpha}A^k_i\,\text{div}_\eta\partial_l^n\eta-{J}^{-1/\alpha}A^k_r
\,[\text{Curl}_\eta\partial_l^n\eta]^r_i \\
&\quad-\sum_{p=1}^{n-1}\{\partial_l^p[{J}^{-1/\alpha}A^k_rA^s_i]\,\partial_l^{n-p}\eta^r,_s
+\tfrac{1}{\alpha}\partial_l^p[{J}^{-1/\alpha}A^k_iA^s_r]\,\partial_l^{n-p}\eta^r,_s\}
\end{split}
\end{equation}
It turns out that the curl comes with bad sign in the main energy estimates, and we will obtain the estimates of 
the curl separately from the curl equation \eqref{curl1} or \eqref{curl2} as if it were an ordinary differential equation.

\subsection{The proof of Lemma \ref{lemma1} : Energy estimates for $\mathcal{E}^N$}

We first claim that when taking $\partial_3^n$ of the equation \eqref{deacoustic},
we obtain the following higher order equations of different structures depending on $n$:
\begin{equation}\label{3n}
\begin{split}
 &w^{\alpha+n}\partial_3^n\eta_{tt}^i+
(w^{1+\alpha+n}\,\partial_3^n[A_i^k{J}^{-1/\alpha}]),_k + w^{\alpha+n} I^{0,n}=0
\end{split}
\end{equation}
where $I^{0,n}$ are essentially lower-order terms given inductively as follows: 
\begin{equation}\label{I0n}
\begin{split}
I^{0,0}=0\,;\;I^{0,n}=\partial_3I^{0,n-1}&
 +\partial_3w\cdot \partial_3^{n-1}[A_i^\kappa J^{-1/\alpha}],_\kappa - w,_\kappa\cdot \,\partial_3^{n}[A_i^\kappa J^{-1/\alpha}] \\
 &+(\alpha+n)\partial_3 w,_k\cdot\, \partial_3^{n-1}[A_i^kJ^{-1/\alpha}] 
 \text{ for }n\geq 1
\end{split}
\end{equation}
In order to see that, it is convenient to write the equation \eqref{3n} in the following form:
\[
\begin{split}
 \partial_3^n\eta_{tt}^i&+
w\partial_3^{n+1}[A_i^3{J}^{-1/\alpha}]+
(1+\alpha+ n)\partial_3w\,\partial_3^{n}[A_i^3{J}^{-1/\alpha}]\\
&+w\partial_3^{n}[A_i^\kappa{J}^{-1/\alpha}],_\kappa + (1+\alpha+n)w,_\kappa \partial_3^{n}[A_i^\kappa{J}^{-1/\alpha}] +I^{0,n}=0
\end{split}
\]
This form can be readily verified by the induction on $n$ starting from \eqref{deacoustic}.
The equations for general mixed derivatives $\partial_\beta^m\partial_3^n\eta$ read as follows.
\begin{equation}\label{3nm}
 \begin{split}
 &w^{\alpha+n}\partial_\beta^m\partial_3^n\eta_{tt}^i+
(w^{1+\alpha+n}\,\partial_\beta^m\partial_3^n[A_i^k{J}^{-1/\alpha}]),_k +w^{\alpha+n} I^{m,n}=0
\end{split}
\end{equation}
where $I^{m,n}$ is given inductively as follows:  for $|m|\geq 1$
\begin{equation}\label{Imn}
I^{m,n}=\partial_\beta I^{m-1,n}+ \partial_\beta w\cdot \partial_\beta^{m-1}\partial_3^n [A^k_iJ^{-1/\alpha}],_k +(1+\alpha+n)\partial_\beta w,_k \cdot \partial_\beta^{m-1}\partial_3^n [A^k_iJ^{-1/\alpha}]
\end{equation}
The degeneracy of the equation lies along the normal direction $x_3$.
It is interesting to point out that the more normal derivatives we take, the more degeneracy we have, while the structure of the equation
stays the same for the tangential derivatives $\partial_\beta$.

We now perform the energy estimates of \eqref{3nm}
for $\partial_\beta^m\partial_3^n v$ and $D_\eta\partial_\beta^m\partial_3^n\eta$. By using \eqref{highD}, we rewrite \eqref{3nm} as follows: for $1\leq |m|+n\leq N$,
\[
\begin{split}
&  w^{\alpha+n}\partial_\beta^m\partial_3^n\eta_{tt}^i
-(w^{1+\alpha+n}\,{J}^{-1/\alpha}A^k_r\,[D_\eta\partial_\beta^m\partial_3^n\eta]^i_r),_k
-\tfrac{1}{\alpha}(w^{1+\alpha+n}\,{J}^{-1/\alpha}A^k_i\,\text{div}_\eta\partial_\beta^m\partial_3^n\eta),_k\\
&- (w^{1+\alpha+n}\,{J}^{-1/\alpha}A^k_r\,
[\text{Curl}_\eta\partial_\beta^m\partial_3^n\eta]^r_i),_k \\
&-\sum_{ |p|+q=1}^{|m|+n-1}(w^{1+\alpha+n}\,\{\partial_\beta^p
\partial_3^q[{J}^{-1/\alpha}A^k_rA^s_i]\,\partial_\beta^{m-p}
\partial_3^{n-q}\eta^r,_s
+\tfrac{1}{\alpha}\partial_\beta^p
\partial_3^q[{J}^{-1/\alpha}A^k_iA^s_r]\,\partial_\beta^{m-p}
\partial_3^{n-q}\eta^r,_s\}),_k \\
&+w^{\alpha+n} I^{m,n}=0
\end{split}
\]
Multiplying  by $\partial_\beta^m\partial_3^n\eta_t^i$ and integrating  over $\Omega$, we   get
\begin{equation}\label{EEs}
\begin{split}
 {\frac12\frac{d}{dt}\int {w}^{\alpha+n}|\partial_\beta^m\partial_3^n v|^2dx+
\int {w}^{1+\alpha+n} \partial_\beta^m\partial_3^n\eta_t^i,_k
\cdot
{J}^{-1/\alpha}A^k_r\,[D_\eta\partial_\beta^m\partial_3^n\eta]^i_rdx}&\\
+{\tfrac{1}{\alpha} \int {w}^{1+\alpha+n}
\partial_\beta^m\partial_3^n\eta_t^i,_k\cdot
{J}^{-1/\alpha}A^k_i\,\text{div}_\eta\partial_\beta^m\partial_3^n\eta dx}&\\
+{\int {w}^{1+\alpha+n}
\partial_\beta^m\partial_3^n\eta_t^i,_k\cdot {J}^{-1/\alpha}A^k_r\,
[\text{Curl}_\eta\partial_\beta^m\partial_3^n\eta]^r_i dx}&\\
-\sum_{{|p|+q=1}}^{|m|+n-1}\int
\partial_\beta^m\partial_3^n\eta_t^i\cdot(w^{1+\alpha+n}\,\{\partial_\beta^p
\partial_3^q[{J}^{-1/\alpha}A^k_rA^s_i]\,\partial_\beta^{m-p}
\partial_3^{n-q}\eta^r,_s\quad\; &\\
{\quad\quad\quad\quad\quad\quad\quad+\tfrac{1}{\alpha}\partial_\beta^p
\partial_3^q[{J}^{-1/\alpha}A^k_iA^s_r]\,\partial_\beta^{m-p}
\partial_3^{n-q}\eta^r,_s\}),_k dx} &\\
 + \int 
 w^{\alpha+n}\partial_\beta^m\partial_3^n\eta_t^i\cdot
I^{m.n} dx & 
=0
\end{split}
\end{equation}
We will estimate line by line. The first line in \eqref{EEs} 
gives rise to the $L^2$ integral of
 the full gradient $D_\eta\partial_\beta^m\partial_3^n\eta$ with the weight
$w^{1+\alpha+n}J^{-1/\alpha}$ and commutators involving
$\partial_tA^k_r$ and $\partial_tJ$.
\[
\begin{split}
 &\frac12\frac{d}{dt}\{\int w^{\alpha+n}|\partial_\beta^m\partial_3^n v|^2dx
 +\int w^{1+\alpha+n} J^{-1/\alpha}
|D_\eta\partial_\beta^m\partial_3^n\eta|^2 dx \}\\
&-\int w^{1+\alpha+n}
\partial_\beta^m\partial_3^n\eta^i,_k {J}^{-1/\alpha}{A^k_r}_t\cdot
[D_\eta\partial_\beta^m\partial_3^n\eta]^i_rdx
  + \frac{1}{2\alpha}\int w^{1+\alpha+n} J^{-\frac{1+\alpha}{\alpha}}J_t\,
  |D_\eta\partial_\beta^m\partial_3^n\eta|^2
 dx\\
 &=\frac{d}{dt}\mathcal{E}^{m,n}+(i)
\end{split}
\]
For the second line in \eqref{EEs}, $\text{div}_\eta\partial_\beta^m\partial_3^n\eta$ and commutators come out.
\[
\begin{split}
 &\frac{1}{2\alpha}\frac{d}{dt}\int w^{1+\alpha+n} J^{-1/\alpha}
 |\text{div}_\eta\partial_\beta^m\partial_3^n\eta|^2 dx\\
& -\tfrac{1}{\alpha}
\int w^{1+\alpha+n} \partial_\beta^m\partial_3^n\eta^i,_k
{J}^{-1/\alpha}{A^k_i}_t\cdot \text{div}_\eta \partial_\beta^m\partial_3^n\eta  dx
  + \tfrac{1}{2\alpha^2}\int w^{1+\alpha+n} J^{-\frac{1+\alpha}{\alpha}}J_t
\, |\text{div}_\eta\partial_\beta^m\partial_3^n\eta|^2
 dx\\
 &=\frac{d}{dt}\mathcal{D}^{m,n} +(ii)
\end{split}
\]
The third line can be written as $\text{Curl}_\eta\partial_\beta^m\partial_3^n\eta$ plus commutators.
\[
 \begin{split}
 &-\sum_{i>r} \int w^{1+\alpha+n} {J}^{-1/\alpha} [A^k_i\partial_\beta^m\partial_3^n\eta_t^r,_k -
 A^k_r\partial_\beta^m\partial_3^n\eta_t^i,_k
]\cdot[\text{Curl}_\eta \partial_\beta^m\partial_3^n\eta]^r_i\, dx\\
&=-\frac{1}{4}\frac{d}{dt} \int w^{1+\alpha+n} {J}^{-1/\alpha}
|\text{Curl}_\eta\partial_\beta^m\partial_3^n\eta|^2 dx -\frac{1}{4\alpha} \int w^{1+\alpha+n}
{J}^{-\frac{1+\alpha}{\alpha}}J_t \,
|\text{Curl}_\eta\partial_\beta^m\partial_3^n\eta|^2
dx\\&\quad+
\sum_{i>r} \int w^{1+\alpha+n} {J}^{-1/\alpha} [{A^k_i}_t\,\partial_\beta^m\partial_3^n\eta^r,_k
- {A^k_r}_t\,\partial_\beta^m\partial_3^n\eta^i,_k
]\cdot[\text{Curl}_\eta \partial_\beta^m\partial_3^n\eta]^r_i\, dx\\
&=-\frac{d}{dt}\mathcal{C}^{m,n}+(iii)
 \end{split}
\]
Note that we have used the following anti-symmetrization to obtain
the curl structure. For any $(1,1)$ tensors $E$ and $F$
\[
 \begin{split}
&  \sum_{i,r=1}^3 E^i_r\,(F^r_i-F^i_r) =  \left\{\sum_{i>r} +  \sum_{i<r}\right\} E^i_r\,(F^r_i-F^i_r)\\
&=\sum_{i>r}(E^i_r-E^r_i)\,(F^r_i-F^i_r)+\underbrace{\sum_{i>r}E^r_i\,(F^r_i-F^i_r) +\sum_{i<r}E^i_r\,(F^r_i-F^i_r)}_{=0}\\
&=-\sum_{i>r}(E^r_i-E^i_r)\,(F^r_i-F^i_r)
 \end{split}
\]
Since $|{A_r^k}_t|$ and $|J_t|$ are bounded by the energy $\mathcal{E}^{N}$ given in \eqref{energyf},
the commutators $(i),\,(ii),\,(iii)$ are
 bounded by a continuous function of $\mathcal{E}^{N}$:
\[
  |(i)|+  |(ii)|+  |(iii)|\leq \mathcal{F}_3(\mathcal{E}^{N})
\]

Next we turn into the fourth and fifth lines in \eqref{EEs} which consist of lower order nonlinear terms.
We present the detail only for the
first term and the other can be estimated in the same way.
We consider two cases: $|m|=0$ and $|m|\geq 1$.
For $|m|=0$, the index $p$ does not appear and the first term reads as follows: for
$1\leq q\leq n-1$
\[
\begin{split}
&\int \partial_3^n\eta_t^i\cdot(w^{1+\alpha+n}\,
\partial_3^q[{J}^{-1/\alpha}A^k_rA^s_i]\,\partial_3^{n-q}\eta^r,_s),_kdx\\
&=\int \partial_3^n\eta_t^i\cdot(w^{1+\alpha+n}
\partial_3^q[{J}^{-1/\alpha}A^\kappa_rA^s_i]\,\partial_3^{n-q}\eta^r,_s),_\kappa dx\\
&\quad+
\int \partial_3^n\eta_t^i\cdot(w^{1+\alpha+n}
\partial_3^q[{J}^{-1/\alpha}A^3_rA^s_i]\,\partial_3^{n-q}\eta^r,_s),_3dx\\
&\equiv(iv)+(v)
\end{split}
\]
For $(iv)$ where $\kappa=1,$ $2$
\[
\begin{split}
 (iv)&=\int \partial_3^n\eta_t^i\cdot w^{1+\alpha+n}
\partial_\kappa\partial_3^q[{J}^{-1/\alpha}A^\kappa_rA^s_i]\,\partial_3^{n-q}\eta^r,_s dx\\
&\quad+
\int \partial_3^n\eta_t^i\cdot w^{1+\alpha+n}
\partial_3^q[{J}^{-1/\alpha}A^\kappa_rA^s_i]\,\partial_\kappa\partial_3^{n-q}\eta^r,_s dx
\end{split}
\]
Since $|w^{q/2}\partial_3^{q}\eta^r,_s|$ and
 $|w^{(q-1)/2}\partial_\beta\partial_3^{q-1}\eta^r,_s|$ for each
$ q\leq [n/2]$ are bounded due to the assumption \eqref{Assumption}, $(iv)$ is bounded by
\[
\mathcal{E}^{0,n}+ \sum_{q=1}^{n-1}(\mathcal{E}^{0,q}+\mathcal{E}^{1,q})\leq \mathcal{F}_4(\mathcal{E}^N)\,.
\]
For $(v)$:
\[
\begin{split}
 (v)=\,&(1+\alpha+n)\int w' \partial_3^n\eta_t^i\cdot w^{\alpha+n}
\partial_3^q[{J}^{-1/\alpha}A^3_rA^s_i]\,\partial_3^{n-q}\eta^r,_sdx\\
&+\int \partial_3^n\eta_t^i\cdot w^{1+\alpha+n}
\partial_3^{q+1}[{J}^{-1/\alpha}A^3_rA^s_i]\,\partial_3^{n-q}\eta^r,_sdx\\
&+
\int \partial_3^n\eta_t^i\cdot w^{1+\alpha+n}
\partial_3^q[{J}^{-1/\alpha}A^3_rA^s_i]\,\partial_3^{n+1-q}\eta^r,_sdx
\end{split}
\]
Since $|w^{q/2}\partial_3^{q}\eta^r,_s|$ and
 $|w^{(q-1)/2}\partial_\beta\partial_3^{q-1}\eta^r,_s|$ for each
$ q\leq [n/2]$ are bounded due to \eqref{Assumption},
$(v)$ is bounded by
\[
\sum_{q=1}^{n}\mathcal{E}^{0,q}\leq \mathcal{F}_5(\mathcal{E}^N)\,.
\]
Now consider $|m|\geq 1$. Recall that $|m|+n\leq N$,
$1\leq |p|+q\leq |m|+n-1$, $0\leq |p|\leq |m|$ and $0\leq q\leq n$. We first integrate by parts in each $x_k$
direction of the second factor and then integrate by parts in $x_\beta$ ($\beta=1\text{ or }2$) direction of
the first factor.
\[
 \begin{split}
 & \int \partial_\beta^m\partial_3^n\eta_t^i\cdot(w^{1+\alpha+n}\,\partial_\beta^p
\partial_3^q[{J}^{-1/\alpha}A^k_rA^s_i]\,\partial_\beta^{m-p}
\partial_3^{n-q}\eta^r,_s),_kdx \\
&=\int \partial_\beta^{m-1}\partial_3^n\eta_t^i,_k\cdot\,w^{1+\alpha+n}\,\partial_\beta^{p+1}
\partial_3^q[{J}^{-1/\alpha}A^k_rA^s_i]\,\partial_\beta^{m-p}
\partial_3^{n-q}\eta^r,_sdx\\
&+\int \partial_\beta^{m-1}\partial_3^n\eta_t^i,_k\cdot\,w^{1+\alpha+n}\,\partial_\beta^{p}
\partial_3^q[{J}^{-1/\alpha}A^k_rA^s_i]\,\partial_\beta^{m+1-p}
\partial_3^{n-q}\eta^r,_sdx
 \end{split}
\]
Note that for $k=1\text{ or } 2$, the first factor together with the weight $w^{(\alpha+n)/{2}}$ is
bounded by $\mathcal{E}^{m,n}$ and for $k=3$, the first factor with the weight $w^{(\alpha+n+1)/{2}}$
is bounded by  $\mathcal{E}^{m-1,n+1}$. Thus from \eqref{Assumption} again, the above is bounded by
\[
\mathcal{E}^{m,n}+ \underbrace{\mathcal{E}^{m-1,n+1}}_{(\ast)}+ \sum_{\substack{p+q=1;\\0\leq p\leq m,0\leq q\leq n}}^{m+n-1}\mathcal{E}^{p+1,q}
\;\leq \;\mathcal{F}_6(\mathcal{E}^N)\,.
\]
Note that since $|m|\geq 1$, $n+1\leq N$ and $|m|+n\leq N$, $(\ast)$ is bounded by the energy
$\mathcal{E}^N$.

For the last line in \eqref{EEs}, first note that $I^{m,n}$ in \eqref{I0n} and \eqref{Imn} consist of terms having at most $|m|+n$ derivatives of $D\eta$ with appropriate weights.  Hence as done for the previous terms, by using \eqref{Assumption} and the fact of $Dw\in X^{\alpha,N}$, it is easy to see that it is bounded by
\[
 \mathcal{E}^{m,n}+\mathcal{E}^{m+1,n-1}+\mathcal{E}^{m-1,n}+ \mathcal{E}^{m,n-1}
\]
This completes the proof of Lemma \ref{lemma1}.

\subsection{The proof of Lemma \ref{lemma2} : Energy estimates for $\mathcal{C}^N$}

It is well known that in Eulerian coordinates, the curl of a
gradient vector is zero. Here is the Lagrangian version.

\begin{claim} If $\omega^k= A^r_k h,_r$,
\[
\text{curl}_\eta \omega =0
\]
\end{claim}
\begin{proof}
\[
\begin{split}
[\text{curl}_\eta \omega]^i &= \epsilon_{ijk} A^s_j\,\omega^k,_s
=\epsilon_{ijk} A^s_j(A^r_k,_sh,_r+A^r_kh,_{rs})\\&=
\epsilon_{ijk} (-A^s_jA^r_m \eta^m,_{sn}A^n_kh,_r+A^s_jA^r_kh,_{rs})
\end{split}
\]
Since both terms are symmetric in $j,\,k$, the conclusion follows.
\end{proof}

The equation for $v$ in \eqref{Lag} is equivalent to
\[
v^i_t+ (1+\alpha) A^k_i (w J^{-1/\alpha}),_k=0
\]
and therefore, by the above claim, the curl
equation is written as follows:
\begin{equation}\label{curl1}
\text{curl}_\eta\partial_t v =0\quad (\epsilon_{ijk} A^s_j\,\partial_tv^k,_s=0)
\end{equation}
It can be also written
\begin{equation}\label{curl2}
\partial_t[\text{curl}_\eta v]^i =\epsilon_{ijk} \partial_tA^s_j\,v^k,_s
\end{equation}
In order to estimate $\mathcal{C}^{m,n}$, we first integrate
\eqref{curl2} in time:
\[
\text{curl}_\eta \partial_t\eta = \text{curl}\, u_0+\epsilon_{ijk}
\int_0^t
\partial_tA^s_j\,v^k,_s d\tau
\]
For $0<|m|+n\leq N$, we take $\partial_\beta^m\partial_3^n$ first to get
\begin{equation}
\begin{split}\label{curleta0}
  \partial_t[\text{curl}_\eta\partial_\beta^m\partial_3^n\eta]^i &=
  [\text{curl}\, \partial_\beta^m\partial_3^n u_0]^i\\
&+ \epsilon_{ijk}\partial_tA^s_j\cdot\partial_\beta^m\partial_3^n{\eta^k},_s- \sum_{|p|+q\geq 1}
\epsilon_{ijk}
 \partial_\beta^p\partial_3^qA^s_j\cdot\partial_\beta^{m-p}\partial_3^{n-q}\partial_t\eta^k,_s\\
&+\sum_{|p|+q\geq 0} \epsilon_{ijk}\int_0^t\partial_\beta^p\partial_3^q
\partial_tA^s_j\cdot
\partial_\beta^{m-p}\partial_3^{n-q}v^k,_s d\tau
 \end{split}
\end{equation}
The last term in \eqref{curleta0} seems to contain more derivatives than the total energy. However, since each term is integrated in time, 
by the fundamental theorem of Calculus, we can move a time derivative of one factor to the other. 
If $|p|+q\geq [N/2]$, we write it as
\[
\begin{split}
\int_0^t\partial_\beta^p\partial_3^q
\partial_tA^s_j\cdot
\partial_\beta^{m-p}\partial_3^{n-q}v^k,_s d\tau&= \partial_\beta^p\partial_3^q
A^s_j\cdot
\partial_\beta^{m-p}\partial_3^{n-q}v^k,_s - \partial_\beta^p\partial_3^q
A^s_j\cdot
\partial_\beta^{m-p}\partial_3^{n-q}v^k,_s(0)\\
&\quad-\int_0^t \partial_\beta^p\partial_3^q A^s_j\cdot 
\partial_\beta^{m-p}\partial_3^{n-q}\partial_tv^k,_s d\tau
\end{split}
\]
and if $|p|+q\leq [N/2]$,
\[
\begin{split}
\int_0^t\partial_\beta^p\partial_3^q
\partial_tA^s_j\cdot 
\partial_\beta^{m-p}\partial_3^{n-q}v^k,_s d\tau&= \partial_\beta^p\partial_3^q
\partial_tA^s_j\cdot 
\partial_\beta^{m-p}\partial_3^{n-q}\eta^k,_s - \partial_\beta^p\partial_3^q
\partial_tA^s_j\cdot 
\partial_\beta^{m-p}\partial_3^{n-q}\eta^k,_s(0)\\
&\quad-\int_0^t \partial_\beta^p\partial_3^q \partial_t^2A^s_j\cdot 
\partial_\beta^{m-p}\partial_3^{n-q}\eta^k,_s d\tau
\end{split}
\]
Note that $\partial_\beta^p\partial_3^q
A^s_j(0)=0$ for  $|p|+q\geq [N/2]$ and $\partial_\beta^{m-p}\partial_3^{n-q}\eta^k,_s(0)=0$ for 
$|p|+q\leq [N/2]$ since $[D\eta] (0)=I$, and thus there will be no contribution of initial data. Hence, 
the curl equation \eqref{curleta0} can be written as 
\begin{equation}
\begin{split}\label{curleta}
  \partial_t[\text{curl}_\eta&\partial_\beta^m\partial_3^n\eta]^i=
  [\text{curl}\, \partial_\beta^m\partial_3^n u_0]^i+
2\epsilon_{ijk}
\partial_tA^s_j\cdot\partial_\beta^m\partial_3^n\eta^k,_s\\
&+\sum_{0<|p|+q<[{N}/{2}]}\epsilon_{ijk}(\partial_\beta^p\partial_3^q
\partial_tA^s_j\cdot 
\partial_\beta^{m-p}\partial_3^{n-q}\eta^k,_s- \partial_\beta^p\partial_3^qA^s_j\cdot\partial_\beta^{m-p}\partial_3^{n-q}\partial_t\eta^k,_s) \\
&-\sum_{|p|+q\geq [{N}/{2}]}\epsilon_{ijk} \int_0^t \partial_\beta^p\partial_3^q A^s_j\cdot 
\partial_\beta^{m-p}\partial_3^{n-q}\partial_t^2\eta^k,_s d\tau\\
&-\sum_{|p|+q<[{N}/{2}]}\epsilon_{ijk}
 \int_0^t \partial_\beta^p\partial_3^q \partial_t^2A^s_j\cdot 
\partial_\beta^{m-p}\partial_3^{n-q}\eta^k,_s d\tau\end{split}
\end{equation} 
Note that each factor in the right-hand side has no more than $|m|+n+1$ derivatives of $\eta$. Hence, 
in multiplying \eqref{curleta} by
 $w^{1+\alpha+n}J^{-1/\alpha}[\text{curl}_\eta\partial_\beta^m\partial_3^n\eta]^i$, 
integrating over $\Omega$, and using the a priori assumption \eqref{Assumption} as
done in the previous energy estimates, we obtain the desired energy inequality \eqref{e2} in Lemma \ref{lemma2}. 

The energy bound \eqref{b} for $\mathcal{B}^N(v)$  is obtained directly from \eqref{curleta} since 
$ \partial_t[\text{curl}_\eta\partial_\beta^m\partial_3^n\eta]^i= [\text{curl}_\eta\partial_\beta^m\partial_3^nv]^i+\epsilon_{ijk}
\partial_tA^s_j\cdot\partial_\beta^m\partial_3^n\eta^k,_s $ and each term in the right-hand side of 
\eqref{curleta} is bounded by $\mathcal{E}^N(\eta,v)$ and $C_0$.  This completes the proof of 
Lemma \ref{lemma2}. 

\section{Existence proof} \label{exist-proof}

\subsection{Iteration Scheme}\label{IS}

In this section, we implement the linear approximate scheme and prove that the linear system is 
well-posed in some energy space.

Let the initial data $\eta(0,x)=\eta_0=x$ and $\partial_t\eta(0,x)=u_0(x)$ of the Euler equation \eqref{deacoustic} be given so that $\mathcal{TE}^N(0)\leq B$ for a constant $B>0$ in 
\eqref{ie}.  We will construct 
approximate solutions $\eta_\nu(t,x)$ and $\partial_t\eta_\nu(t,x)$ for each nonnegative integer 
$\nu$, by induction satisfying 
the following properties: 
\begin{equation}
\begin{split}\label{Assume1}
 \eta_\nu(0,x)=x\,, \;\partial_t\eta_\nu(0,x)=u_0(x)\,,\;
 \|A_\nu-I \|_\infty\leq C_\nu\text{ for }C_\nu \leq  1/8
 \end{split}
\end{equation}
as well as 
\begin{equation}\label{Assume2}
 \sum_{|p|+q=0}^{[N/2]} |w^{q/2}\partial_\beta^p\partial_3^q{\eta_\nu}^r,_s|
+ \sum_{|p|+q=0}^{[N/2]-1}|w^{q/2}\partial_\beta^p\partial_3^q{\partial_t\eta_\nu}^r,_s|<\infty\,
\end{equation}
where  the approximate inverse of deformation tensor and Jacobian determinant are defined by
\begin{equation}\label{n}
 A_{\nu}\equiv [D\eta_{\nu}]^{-1}\;;\; J_{\nu}\equiv \det D\eta_{\nu}\,.
\end{equation}
The condition $\|A_\nu-I \|_\infty\leq C_\nu\text{ for }C_\nu \leq  1/8$ in \eqref{Assume1} also guarantees  the non-degeneracy of the approximate flow map:
$$2/3\leq \frac{1}{1+3C_\nu+6C_\nu^2+6C_\nu^3}\leq J_\nu \leq \frac{1}{1-3C_\nu-6C_\nu^2-6C_\nu^3} \leq 2$$ 
The differentiation of $A_\nu$ and $J_\nu$, which is similar to \eqref{DA} and \eqref{DJ},  is given by 
\begin{equation}\label{Dn}
 \begin{split}
 \partial_t{A_{\nu}}^k_i=-{A_{\nu}}^k_r\partial_t{\eta_\nu}^r,_s{A_{\nu}}^s_i\;&;\;
  \partial_l{A_{\nu}}^k_i=-{A_{\nu}}^k_r\partial_l{\eta_\nu}^r,_s{A_{\nu}}^s_i\\
\partial_t{J_{\nu}}=J_\nu{A_{\nu}}^s_r\partial_t{\eta_\nu}^r,_s\;&;\; \partial_l{J_{\nu}}=J_\nu{A_{\nu}}^s_r\partial_l{\eta_\nu}^r,_s
 \end{split}
\end{equation}
Note that Piola identity \eqref{Piola} also holds for $A_\nu$: ${a_\nu}_i^k,_k=0$. 
Having $A_\nu$ and $J_\nu$ defined, we now introduce the following vector-valued, non-negative and weighted second order linear operators $\mathcal{L}^e_\nu$, $\mathcal{L}^d_\nu$ for a given vector $G$ and a weighted first order linear operator $\mathcal{L}^c_\nu$ for a given anti-symmetric matrix $H$ as follows:
\begin{equation}
 \begin{split}\label{L}
[ \mathcal{L}^e_\nu\,G]^i&\equiv -(w^{1+\alpha}\,{J_\nu}^{-1/\alpha}{A_\nu}^k_r{A_\nu}^s_r G^i,_s),_k\\
[\mathcal{L}^d_\nu\,G]^i&\equiv-\tfrac{1}{\alpha}(w^{1+\alpha}\,{J_\nu}^{-1/\alpha}{A_\nu}^k_i
{A_\nu}^s_r{G}^r,_s),_k\\
[\mathcal{L}^c_\nu\,H]^i&\equiv -(w^{1+\alpha}\,{J_\nu}^{-1/\alpha}{A_\nu}^k_rH^r_i),_k
 \end{split}
\end{equation}
For the iteration scheme, instead of $\eta$, we approximate $[-\partial_1^2-\partial_2^2-w^{-\alpha}\partial_3\,w^{1+\alpha}\partial_3+\lambda ]\,\eta\equiv G$ for a sufficiently large positive constant $\lambda>0$ and $\text{Curl}_\eta G=H$.  
For each $\nu\geq 0$, consider the following approximate system for $G_{\nu+1}$ and $H_{\nu+1}$: 
\begin{equation}\label{approx}
 \begin{split}
 w^\alpha\partial_t^2{G_{\nu+1}}
+\mathcal{L}^e_\nu\,{G_{\nu+1}}
+\mathcal{L}^d_\nu\,{G_{\nu+1}}
+ \mathcal{L}^c_\nu\,{H_{\nu+1}} &= w^\alpha\mathcal{R}_\nu \\
\partial_t H_{\nu+1} &= \mathcal{S}_\nu
 \end{split}
\end{equation}
where $\mathcal{R}_\nu$ and $\mathcal{S}_\nu$ are obtained by replacing $\eta$ by $\eta_\nu$ in $\mathcal{R}$ and $\mathcal{S}$ given in \eqref{G} and \eqref{H}.  For the convenience of readers, we put the formal derivation of $G$ and $H$ in Section \ref{GH}. 
Note that $\mathcal{R}_\nu$ and $\mathcal{S}_\nu$ consist of  lower order terms and moreover, 
$$\|\mathcal{R}_\nu \|_{X^{\alpha,N-2}}<\infty\text{ and }\|\mathcal{S}_\nu\|_{Z_\nu^{\alpha,N-2}}<\infty\text{ if }\mathcal{E}^N(\eta_\nu,\partial_t\eta_\nu)<\infty.$$ And also note that $[\mathcal{S}_\nu]^k_j=-[\mathcal{S}_\nu]^j_k$ and thus $H_{\nu+1}$ is  anti-symmetric. 

The initial  data for $G_{\nu+1}$ and $H_{\nu+1}$ are inherited from the original system: 
\begin{equation}
\begin{split}
& {G_{\nu+1}}^i(0,x)=-(1+\alpha)\partial_3w\delta^i_3 +\lambda x^i,\\
& \partial_tG_{\nu+1}(0,x)=[-\partial_1^2-\partial_2^2-w^{-\alpha}
\partial_3w^{1+\alpha}\partial_3+\lambda]u_0(x)\,,
 \,H_{\nu+1}(0,x)=0\,.
 \end{split}\label{id}
\end{equation}
The unique $\eta_{\nu+1}$ is to be found by solving a degenerate elliptic equation 
\begin{equation}\label{de}
 [-\partial_1^2-\partial_2^2-w^{-\alpha}\partial_3\,w^{1+\alpha}\partial_3+\lambda]\,\eta_{\nu+1}= G_{\nu+1}
\end{equation}
 The same goes for $\partial_tG_{\nu+1}$. $A_{\nu+1}$ and $J_{\nu+1}$ are defined as in \eqref{n}. 
We remark that the approximate systems 
\eqref{approx} converge to the equations for $G$ and $\text{Curl}_\eta G$ respectively in the formal limit $\nu\rightarrow\infty $.

We define the approximate energy functional $\tilde{\mathcal{E}}_{\nu+1}$ at the $\nu^{\text{th}}$ step: 
\begin{equation}
\begin{split}\label{AE}
\tilde{\mathcal{E}}_{\nu+1}(t)&\equiv\sum_{|m|+n=0}^{N-2}
\{\frac{1}{2}\int w^{\alpha+n}|\partial_\beta^m\partial_3^n\partial_tG_{\nu+1}|^2dx+
\frac12\int  w^{1+\alpha+n} 
J_\nu^{-\frac{1}{\alpha}}|D_{\eta_\nu}\partial_\beta^m\partial_3^nG_{\nu+1}|^2 dx  \\
&\quad\quad\quad\quad\;+\frac{1}{2\alpha}\int w^{1+\alpha+n} 
J_\nu^{-\frac{1}{\alpha}}|\text{div}_{\eta_\nu}\partial_\beta^m\partial_3^nG_{\nu+1}|^2 dx 
\\&\quad\quad\quad\quad\;-\frac{1}{2}\int w^{1+\alpha+n} 
J_\nu^{-\frac{1}{\alpha}} \text{Curl}_{\eta_\nu} \partial_\beta^m\partial_3^nG_{\nu+1} \cdot 
\partial_\beta^m\partial_3^nH_{\nu+1} dx\\
&\quad\quad\quad\quad\;+ 2 \int w^{1+\alpha+n} 
J_\nu^{-\frac{1}{\alpha}} |\partial_\beta^m\partial_3^nH_{\nu+1}|^2dx  \}\\
&\equiv \;\frac12\|\partial_tG_{\nu+1}\|^2_{X^{\alpha,N-2}}+ \frac12\|G_{\nu+1}\|^2_{Y_\nu^{\alpha,N-2}}+ 2\|H_{\nu+1}\|^2_{Z_\nu^{\alpha,N-2}}\\
&\quad+\sum_{|m|+n=0}^{N-2}\{\mathcal{D}_\nu^{m,n}(G_{\nu+1})
-\frac{1}{2}\int w^{1+\alpha+n} 
J_\nu^{-\frac{1}{\alpha}} \text{Curl}_{\eta_\nu} \partial_\beta^m\partial_3^nG_{\nu+1} \cdot \partial_\beta^m\partial_3^nH_{\nu+1} dx\}
\end{split}
\end{equation}
where $Y_\nu^{\alpha,N-2}$ and $Z_\nu^{\alpha,N-2}$ denote $Y^{\alpha,N-2}$ and $Z^{\alpha,N-2}$ in  \eqref{XY} induced by $\eta_\nu$. Note that by Cauchy-Schwartz inequality, 
\[
\begin{split}
&|\frac12\int w^{1+\alpha+n} 
J_\nu^{-\frac{1}{\alpha}} \text{Curl}_{\eta_\nu} \partial_\beta^m\partial_3^nG_{\nu+1} \cdot 
\partial_\beta^m\partial_3^nH_{\nu+1} dx|\\
&\leq \frac{1}{3} \int  w^{1+\alpha+n} 
J_\nu^{-\frac{1}{\alpha}}|D_{\eta_\nu}\partial_\beta^m\partial_3^nG_{\nu+1}|^2 dx  
+\frac32  \int w^{1+\alpha+n} 
J_\nu^{-\frac{1}{\alpha}} |\partial_\beta^m\partial_3^nH_{\nu+1}|^2dx
\end{split}
\]
and therefore
\begin{equation}
\begin{split}\label{ee}
\frac12\|\partial_tG_{\nu+1}\|^2_{X^{\alpha,N-2}}+ \frac16\|G_{\nu+1}\|^2_{Y_\nu^{\alpha,N-2}}+ 
\frac12\|H_{\nu+1}\|^2_{Z_\nu^{\alpha,N-2}}\leq \tilde{\mathcal{E}}_{\nu+1}&\\
\leq \frac12\|\partial_tG_{\nu+1}\|^2_{X^{\alpha,N-2}}+ (\frac56+\frac{1}{2\alpha})\|G_{\nu+1}\|^2_{Y_\nu^{\alpha,N-2}}&+ \frac72\|H_{\nu+1}\|^2_{Z_\nu^{\alpha,N-2}} 
\end{split}
\end{equation}

We now state and prove that the approximate system \eqref{approx} is well-posed in the energy 
space generated by ${\mathcal{E}}_\nu$ under the following induction hypotheses: \\

(HP1) $\tilde{\mathcal{E}}_\nu<\infty$ and $\eta_\nu$ and $\partial_t\eta_\nu$ satisfy \eqref{Assume1}.

(HP2) The left-hand side of \eqref{Assume2} is bounded by $\tilde{\mathcal{E}}_\nu$ and 
 $\tilde{\mathcal{E}}_{\nu-1}$.  

(HP3) $\mathcal{E}^N(\eta_\nu,\partial_t\eta_\nu)$ is bounded by $\tilde{\mathcal{E}}_\nu$ and $\tilde{\mathcal{E}}_{\nu-1}$.  \\

\begin{proposition}[Well-Posedness of Approximate system and Regularity]\label{prop2} Under hypotheses 
(HP1) - (HP3), linear system \eqref{approx} admits a unique solution 
$(\partial_tG_{\nu+1},G_{\nu+1}, H_{\nu+1})$ in $X^{\alpha,N-2}, Y_\nu^{\alpha,N-2}, Z_\nu^{\alpha,N-2}$. Furthermore, we obtain the following energy bounds: 
\[
\tilde{\mathcal{E}}_{\nu+1}(t)\leq \tilde{\mathcal{E}}_{\nu+1}(0)+  \int_0^t \mathcal{F}_7
(\tilde{\mathcal{E}}_{\nu+1},\tilde{\mathcal{E}}_\nu,
 \tilde{\mathcal{E}}_{\nu-1},B^N(u_0))(\tilde{\mathcal{E}}_{\nu+1})^{\frac12}d\tau
\]
where $\mathcal{F}_7
(\tilde{\mathcal{E}}_{\nu+1},\tilde{\mathcal{E}}_\nu,
 \tilde{\mathcal{E}}_{\nu-1},B^N(u_0))$ is a continuous function of $
\tilde{\mathcal{E}}_{\nu+1},\,\tilde{\mathcal{E}}_\nu,\,
 \tilde{\mathcal{E}}_{\nu-1},\,B^N(u_0)$. 
\end{proposition}

Note that $\mathcal{R}_\nu\in Y^{\alpha,N-2}_\nu$ and $\mathcal{S}_\nu\in Z_\nu^{\alpha,N-2}$ under (HP3).  Proposition \ref{prop2} directly follows from Proposition \ref{dual} given in Section \ref{duality}.

In order to complete the induction procedure of approximate schemes, it now remains to verify the induction hypotheses (HP1), (HP2), and (HP3) for $\nu+1$.  
 
 By the following Elliptic Regularity lemma, which will be proven in Section \ref{proof-er}, we obtain a unique $\eta_{\nu+1}$ to the degenerate 
 elliptic equation \eqref{de}.  

\begin{lemma}\label{er} Let $k\geq 0$ be given. For each $G\in X^{\alpha,k}$, there exists a unique solution  $u\in X^{\alpha,k+2}$ to the following degenerate elliptic equation 
\[
[-\partial_1^2-\partial_2^2-w^{-\alpha}\partial_3w^{1+\alpha}\partial_3+\lambda]u= G
\]
Moreover, we have 
\[
\|u\|_{X^{\alpha,k+2}}\precsim 
\|G\|_{X^{\alpha,k}}\,.
\]
\end{lemma}

First (HP1). By Lemma \ref{er}, $\eta_{\nu+1}$ and $\partial_t\eta_{\nu+1}$ constructed in the above satisfy the following 
\[
\|\partial_t\eta_{\nu+1}\|_{X^{\alpha,N}}+\|\eta_{\nu+1}\|_{Y_\nu^{\alpha,N}}\precsim 
\|\partial_tG_{\nu+1}\|_{X^{\alpha,N-2}}+\|G_{\nu+1}\|_{Y_\nu^{\alpha,N-2}}
\]
as well as  the initial boundary conditions in \eqref{Assume1}.  The boundedness of $J_{\nu+1}$ will follow from the continuity argument by using the estimate of $\partial_tJ_{\nu+1}$.  Note that from Jacobi's formula 
$$\partial_tJ_{\nu+1}=\partial_t\det D\eta_{\nu+1}=\text{tr}(\text{adj}(D\eta_{\nu+1})\partial_t
D\eta_{\nu+1})$$ where $\text{adj}(\cdot)$ denotes the adjugate of a given matrix. Since 
$\partial_t\eta_{\nu+1}\in X^{\alpha,N}$ and $\eta_{\nu+1}\in Y_\nu^{\alpha,N}$, by Lemma \ref{emb}, $|D\eta_{\nu+1}|$ and $|\partial_tD\eta_{\nu+1}|$ are bounded by their $X^{\alpha,N}$ and $Y_\nu^{\alpha,N}$ norms and hence bounded by $\tilde{\mathcal{E}}_{\nu+1}$, $\tilde{\mathcal{E}}_{\nu}$. Thus
\[
|\partial_tJ_{\nu+1}|\leq \tilde{C}_{\nu+1}
\]
where $\tilde{C}_{\nu+1}$ depends only on $\tilde{\mathcal{E}}_{\nu+1}$, $\tilde{\mathcal{E}}_{\nu}$, and initial data. Since $J_{\nu+1}(t)=J_{\nu+1}(0)+\int_0^t\partial_tJ_{\nu+1} d\tau $, we get  
\[
1-\tilde{C}_{\nu+1}T \leq J_{\nu+1} \leq 1+ \tilde{C}_{\nu+1}T
\]
Now since $J_{\nu+1}$ is bounded away from zero for sufficiently small $T$, $D\eta_{\nu+1}$ is invertible and thus $A_{\nu+1}$ is well-defined and the formulation in \eqref{Dn} is also well-defined. To verify the last condition in \eqref{Assume1} for $(\nu+1)^{\text{th}}$, we use \eqref{Dn}. Since $A_{\nu+1}$ and $\partial_t D \eta_{\nu+1}$ are bounded by $\tilde{\mathcal{E}}_{\nu+1}$, $\tilde{\mathcal{E}}_{\nu}$ and since $A_{\nu+1} (t)=A_{\nu+1}(0)+\int_0^t \partial_tA_{\nu+1} d\tau$, 
\[
\|A_{\nu+1}-I\|_\infty \leq \tilde{\tilde{C}}_{\nu+1}T
\] 
and thus for sufficiently small $T$, $C_{\nu+1}\leq 1/8$ can be found. 

We move onto (HP2). This directly follows from the embedding as in Lemma \ref{emb}, since $N$ is sufficiently large 
$N\geq 2 [\alpha]+9$. 

For (HP3), it suffices to show that for each $m+n\leq N$, $\int_\Omega w^{1+\alpha+n} 
J_{\nu+1}^{-1/\alpha} |D_{\eta_{\nu+1}}\partial_\beta^m\partial_3^n \eta_{\nu+1}|^2 dx$ is bounded 
by $\tilde{\mathcal{E}}_{\nu+1}$ and $\tilde{\mathcal{E}}_{\nu}$. Note that 
\[
\begin{split}
&\int_\Omega w^{1+\alpha+n} 
J_{\nu+1}^{-1/\alpha} |D_{\eta_{\nu+1}}\partial_\beta^m\partial_3^n \eta_{\nu+1}|^2 dx \\
&= \int_\Omega w^{1+\alpha+n} 
(\frac{J_{\nu+1}}{J_\nu})^{-1/\alpha} J_{\nu}^{-1/\alpha} |{A_{\nu+1}}^s_r {\eta_{\nu}},_s^k  {A_{\nu}}^s_k\partial_\beta^m\partial_3^n {\eta_{\nu+1}}^i,_s|^2 dx
\end{split}
\]
Because $J_{\nu+1}$ is bounded away from zero and $A_{\nu+1}$ and $D\eta_\nu$ are bounded by 
 $\tilde{\mathcal{E}}_{\nu+1}$ and $\tilde{\mathcal{E}}_{\nu}$, (HP3) at $(\nu+1)^{\text{th}}$ step 
  readily follows. 

\subsection{Convergence of the iteration scheme and uniqueness of solution}\label{5}

In order to prove Theorem \ref{thm-general}, it now remains to show that $\eta_\nu$ and $\partial_t\eta_\nu$ converge, the limit functions solve Euler equations \eqref{eta} and \eqref{deacoustic}, and they are unique. 

First, by applying the Gronwall inequality to the energy inequality obtained in Proposition \ref{prop2}, we can deduce the following Claim: 

\begin{claim}\label{claim1}
Suppose that the initial data $\eta(0,x)=x$ and $\partial_t\eta(0,x)=u_0(x)$ of Euler equations 
\eqref{deacoustic} are given such that $\mathcal{TE}^N(0)\leq B$ for a constant $B>0$. Then 
there exist $T>0$ such that if for all $\iota\leq\nu$, $\tilde{\mathcal{E}}_\iota\leq 3B/2$ for $t\leq T$, then $\tilde{\mathcal{E}}_{\nu+1}\leq 3B/2$ for $t\leq T$ and in addition, for all $\nu$,   $\| A_{\nu}-I\|_\infty \leq 1/8$ and $2/3\leq J_\nu\leq 2$. 
\end{claim}

Thus we get the uniform bound of $\tilde{\mathcal{E}}_{\nu}$ as well as the uniform upper and lower bounds of $J_\nu$. In order to take the limit $\nu\rightarrow \infty$, in view of \eqref{ee}, we define the homogenous energy functional $\overline{\mathcal{E}}_{\nu+1}$ for $G_{\nu+1}$ and $H_{\nu+1}$: 
\begin{equation}\label{hef}
\begin{split}
\overline{\mathcal{E}}_{\nu+1}\equiv \sum_{|m|+n=0}^{N-2}&
\{\frac{1}{2}\int w^{\alpha+n}|\partial_\beta^m\partial_3^n\partial_tG_{\nu+1}|^2dx+
\frac16\int  w^{1+\alpha+n} |\partial_\beta^m\partial_3^nDG_{\nu+1}|^2 dx \\
&+ \frac12 \int w^{1+\alpha+n}  |\partial_\beta^m\partial_3^nH_{\nu+1}|^2dx  \}
\end{split}
\end{equation}
Then due to \eqref{ee} and since   $\| A_{\nu}-I\|_\infty \leq 1/8$ and $2/3\leq J_\nu\leq 2$, $\tilde{\mathcal{E}}_{\nu+1}$ and $\overline{\mathcal{E}}_{\nu+1}$ are equivalent: 
\[
\frac{1}{1+M_\nu}\tilde{\mathcal{E}}_{\nu+1}\leq \overline{\mathcal{E}}_{\nu+1}\leq (1+M_\nu)\tilde{\mathcal{E}}_{\nu+1}
\] 
where $M_\nu$ depends only on $C_\nu$ and $\tilde{\mathcal{E}}_{\nu}$. Now by Claim \ref{claim1}, $\tilde{\mathcal{E}}_{\nu+1}$ and $M_\nu$'s have the uniform bound over $t\leq T$. 
Therefore, there exists a sequence $\nu_\iota$ such that $G_{\nu_\iota}$, $H_{\nu_\iota}$, and $\eta_{\nu_\iota}$ converge strongly  to some $G,H,\eta$.  
Due to the uniform energy bound, we also conclude that $\eta$ and $\partial_t\eta$ solve 
\eqref{eta} and \eqref{deacoustic} with the desired properties.

For uniqueness, let $(\eta, v)$ and $(\overline{\eta},\overline{v})$ be two 
solutions to \eqref{eta} and \eqref{deacoustic} with the same initial boundary conditions having the total energy bounds:  
$\mathcal{TE}^N(\eta,v),$ $\mathcal{TE}^N(\overline{\eta},\overline{v})\leq 2B$.  Define  
$\mathcal{Z}(t)$ by
\[
\begin{split}
\mathcal{Z}(t)&\equiv \frac12\int  w^\alpha |v-\overline{v}|^2 dx+ \alpha \int w^{1+\alpha} 
J^{-1/\alpha-2}|J-\overline{J}|^2 dx\\
&+\sum_{|m|+n=1}^{N-1}\frac12\int w^{\alpha+n}|\partial_\beta^m\partial_3^n( v-\overline{v})|^2dx
+\frac12\int w^{1+\alpha+n}J^{-1/\alpha}|D_\eta\partial_\beta^m\partial_3^n
( \eta-\overline{\eta})|^2dx 
\end{split}
\]
From \eqref{3nm}, the equations for $\eta-\overline{\eta}$ read as 
\[
\begin{split}
 w^{\alpha+n}\partial_\beta^m\partial_3^n(\eta-\overline{\eta})_{tt}^i+
(w^{1+\alpha+n}\,\partial_\beta^m\partial_3^n[A_i^k{J}^{-1/\alpha}-\overline{A}_i^k{\overline{J}}^{-1/\alpha}]),_k
+ w^{\alpha+n} [I^{m,n} (\eta)-I^{m,n}(\overline{\eta})] =0
\end{split}
\]
In performing the energy estimates of $\eta-\overline{\eta}$ as done in the a priori estimates and noting that  $|J^{-1/\alpha}-\overline{J}^{-1/\alpha}|\leq C_{1B}|D\eta-D\overline{\eta}|$ and $|A-\overline{A}| \leq C_{2B}|D\eta-D\overline{\eta}| $, one can obtain  $$\frac{d\mathcal{Z}}{dt}\leq C\mathcal{Z}\text{ for }C \text{ depending only on }B\,.$$ Since $\mathcal{Z}(0)=0$, this immediately yields the uniqueness. 

\subsection{Proof of Lemma \ref{er}}\label{proof-er}
 
Here we give a sketch of proof of Lemma \ref{er}. We refer to \cite{Baouendi67} and \cite{BG69} for 
similar regularity results of degenerate elliptic problems. We introduce the 
spaces  $H = X^{\alpha,0} $  and  $V = X^{\alpha,1} $ defined by 
\begin{equation*}
V = \{ u \in \D'(\Omega) ; \quad w^{\alpha \over 2} u, \,  w^{\alpha \over 2} \partial_\beta u  \in L^2(\Omega), 
\,   w^{\alpha + 1  \over 2} \partial_3 u  \in L^2(\Omega)
     \}.  
\end{equation*}
We also define the following scalar product on $H$, $(u,v)_H = \int_\Omega w^\alpha uv. $
We define on $V$ the following norm 
 \begin{equation*}
\| u  \|_{V}^2  =  \| w^{\alpha \over 2} u  \|^2_{L^2(\Omega)  }  +  
  \sum_{\beta = 1}^2  \| w^{\alpha \over 2}  \partial_\beta u  \|^2_{L^2(\Omega)  }  + 
 \| w^{\alpha + 1  \over 2} \partial_3 u   \|^2_{L^2(\Omega)  } 
\end{equation*}

\begin{lemma} \label{dens}
$\D$ is dense in $V$. 
\end{lemma}
The proof is based on an  explicit approximation. We define $\chi_n(t)  = 1$  
 if $\frac2n \leq t \leq 1 - \frac2n $ and $\chi_n (t) = n t - 1   $  if $\frac1n \leq t \leq \frac2n $
and $\chi_n (t) = n-1  - nt  $  if $ 1 - \frac2n \leq t \leq 1- \frac1n  $ and 
 $\chi_n(t) = 0 $  for $ 0\leq t \leq \frac1n  $ or  $ 1 - \frac1n \leq t \leq 1  $.  
 For each $v \in V$, we define $u_n = \chi_n(x_3) v $ and $v^n = \rho_{\frac1{2n}} * u_n $  
the convolution of $u_n$ with the mollifier $ \rho_{\frac1{2n}} = (2n)^3 \rho (\frac{.}{2n}) $  
where $\rho \in C^\infty_0(\R^3)$,   $\rho \geq 0$, $\int \rho = 1 $ and 
supp$(\rho) \in B(0,1)$.   Hence $ v_n \in \D(\Omega)  $ and $v_n $ goes to $v$ in 
$V$ when $n$ goes to infinity. To prove that $v_n$ converges to $v$ in $V$, we have to use the Hardy 
inequality, namely the fact that 
  \begin{equation*}
 \int_\Omega  w^{\alpha - 1 } | v  |^2 dx \,  
  \leq      \int_\Omega  w^{\alpha + 1 } |\partial_3 v  |^2 dx.  \,
\end{equation*} 
This ends the proof of lemma \ref{dens}.

We also define   the following bilinear form 
\begin{equation*}
 B[u,v]     =  \int_\Omega  \lambda   w^{\alpha } u v   +  
  \sum_{\beta = 1}^2    \partial_\beta u  \partial_\beta (  w^{\alpha }     v)     + 
  w^{\alpha + 1 } \partial_3 u   \partial_3 v 
\end{equation*}
where $\lambda$ is big enough. 
Notice that $B$ is not symmetric if $w$ depends on the tangential variable. 
If $\lambda$ is chosen big enough then, $B$ satisfies the hypotheses of 
the Lax-Milgram theorem, namely the fact that 
$$  | B[u,v] |  \leq C \| u\|_V  \| v\|_V, \quad \text{and}   \quad      \frac1C    \| v\|_V^2 \leq    B[v,v] .  $$  
Hence, for each bounded linear functional $f$  on $V$, namely $f : V \to \R$, 
there exists a unique element $u \in V$ such that  $ B[u,v] = f(v)   $ for each $v \in V$. 
One can try to characterize the set $V'$, but we do not need to do it here. 
For any $ G \in H$,  $f(v) =  (v,G )_{H} $ defines a linear functional on $V$ and hence 
there exists a unique $u \in V$ such that $  B[u,v]  = \int_\Omega  w^{\alpha }  G v   $ for 
all $v \in V$. 

Using lemma \ref{dens} and using a density argument, it is easy to see that 
$u$ solves 
\begin{equation} \label{deg-pde}
[ - \partial_1^2 - \partial_2^2 -  w^{-\alpha}\partial_3w^{1+\alpha}\partial_3 + \lambda ] u = G
\end{equation}
if and only if $   B[u,v]  = \int_\Omega  w^{\alpha }  G v   $ for all $v \in V$. 
Hence, $u$ is the unique solution of \eqref{deg-pde}.  
To prove that $u$ is more regular, we use again Lax-Milgram to construct 
 $ U_\beta  $ the solution of 
\eqref{deg-pde} with the right hand side $ G_\beta =   \partial_\beta G + \alpha 
 \frac{\partial_\beta w}w   w^{-\alpha}  \partial_3 [ w^{1+\alpha}\partial_3 u ]  + 
 (1+\alpha)  w^{-\alpha} 
  \partial_3  [ w^{\alpha+1} \frac{\partial_\beta w} w  \partial_3 u  ]    $.    Indeed, one has just 
to observe that  $f_\beta (v) =  (v,  G_\beta  )_{H}=$
\[
\begin{split}
 - ( w^{-\alpha } \partial_\beta  (w^\alpha v), G )_{H} 
   - \alpha ( w^{1/2} \partial_3(\frac{\partial_\beta w}w  v   ) , w^{1/2} \partial_3 u    )_{H}   
   - (1 + \alpha)  (w^{1/2} \partial_3 v , w^{1/2}  \frac{\partial_\beta w}w  \partial_3 u     )_{H}   
\end{split}
\]
defines a linear functional of $V$. Then by uniqueness of the solution to \eqref{deg-pde} 
with the right hand side $G_\beta$, we deduce that $U_\beta = \partial_\beta u $. 
To be more precise, one has to replace the partial derivative with respect to $x_\beta$ 
by difference quotient and then pass to the limit to deduce that 
$U_\beta = \partial_\beta u \in V $ and that $ \| \partial_\beta u \|_V \leq C \| G \|_{H} .     $ 
Using the equation  \eqref{deg-pde}, we also see that 
$ w^{-\alpha}\partial_3w^{1+\alpha}\partial_3  u  \in H $ and hence by Hardy inequality, 
we deduce that 
\begin{equation} \label{hard} \int_\Omega w^{-\alpha - 2 }   | w^{1+\alpha} \partial_3 u   |^2 dx \leq C 
  \int_\Omega w^{-\alpha  }   |\partial_3  w^{1+\alpha} \partial_3 u   |^2 dx  \leq C \| G \|_H.   
\end{equation} 
Using that $ w^{-\alpha}\partial_3w^{1+\alpha}\partial_3  u  = (1+ \alpha ) \partial_3 w \partial_3  u 
  +w \partial_3^2  u   $ and the fact that 
\eqref{hard} yields that $\partial_3 u \in H$, 
 we deduce that $ w \partial_3^2  u  \in H   $. 

We now assume that $G$ is more regular, namely that $G \in X^{\alpha,k}$ for some $k \geq 1$.
 We want to prove 
that $ u \in X^{\alpha,k+2 } $. From the previous argument, we know that    $ u \in X^{\alpha, 2 } $. 
Proving regularity in the tangential direction is very similar to the previous argument. 
Indeed, differentiating \eqref{deg-pde} with respect to $\beta$, we see that 
$\partial_\beta u$ solves \eqref{deg-pde} with the right hand side $ G_\beta $ and due to the 
extra regularity of $G$, we see that $ G_\beta \in H $. Hence, we can apply $  X^{\alpha, 2 } $ 
regularity property to $ \partial_\beta u  $ and deduce that it is in $  X^{\alpha, 2 } $. 
Of course we can repeat this  $k$ times  and deduce that the  tangential derivatives  
$\partial_\beta^k  u \in  X^{\alpha, 2 }  $. It remains to control the normal derivatives. 
Taking one normal derivative, we observe that $\partial_3 u  $ solves  
\begin{equation} \label{deg-pde3}
[ - \partial_1^2 - \partial_2^2 -  w^{-\alpha-1}\partial_3w^{2+\alpha}\partial_3 + \lambda ] \partial_3 u = G_3
\end{equation}
with the right hand side $G_3 = \partial_3 G + (1+\alpha) \partial^2_3 w \partial_3 u  $. 
Moreover, it is clear that $ \partial_3 u  \in H $ and that $G_3 \in H $. 
We introduce two  other Hilbert spaces $H_{\alpha+1}$ and $V_{\alpha+1}$   where 
for each $\gamma > 0$, $H_\gamma$ and $V_\gamma$ are given  by the following norms  
  \begin{equation*} 
\| u  \|_{H_\gamma}^2  =  \| w^{\gamma  \over 2} u  \|^2_{L^2(\Omega)  }  
\end{equation*} 
 \begin{equation*} 
\| u  \|_{V_\gamma}^2  =  \| w^{\gamma  \over 2} u  \|^2_{L^2(\Omega)  }  +  
  \sum_{\beta = 1}^2  \| w^{\gamma \over 2}  \partial_\beta u  \|^2_{L^2(\Omega)  }  + 
 \| w^{\gamma  + 1  \over 2} \partial_3 u   \|^2_{L^2(\Omega)  } .  
\end{equation*}
 In particular, we notice that $H=H_\alpha$ and $V = V_\alpha$.

Now, we can apply the previous regularity argument with $H$ replaced by 
$H_{\alpha+1}$ and    $  X^{\alpha, 2 } $  replaced by    $  X^{\alpha+1, 2 } $.
We can also combine the $x_3$ derivatives with tangential derivatives and 
prove that for all $m$, $0\leq m \leq k-1$, we have 
  $ \partial_\beta^{m}    \partial_3 u  \in X^{\alpha+1,2} $. 

By an induction argument on the number of $x_3$ derivatives, we can finally 
prove that   for all $m$, $0\leq m \leq k-n$, we have 
  $ \partial_\beta^{m}    \partial_3^n  u  \in X^{\alpha+n,2} $.
Hence, we deduce that 
 \begin{equation*} 
\| u  \|_{X^{\alpha,k+2}}  \leq  C \| G   \|_{X^{\alpha,k}}.  
\end{equation*}

\subsection{Duality Argument: solvability for \eqref{approx}}\label{duality}

\begin{proposition}\label{dual} Let $\eta$ and $\partial_t\eta$ be given such that $\mathcal{E}^N(\eta,\partial_t\eta)<\infty$. For $f=f(\eta)$ and $g=g(\eta,\partial_t\eta)$ 
where ${f}\in L^1(0,T;X^{\alpha,0})$, $g\in L^1(0,T; Z^{\alpha,0})$, and $g$ is anti-symmetric, there exists a unique solution $(\partial_tG,G,H)$ on $(0,T)$ to 
the linear system 
\begin{equation}\label{linear}
\begin{cases}
w^\alpha\partial_t^2G+\mathcal{L}^eG+\mathcal{L}^dG+\mathcal{L}^cH=w^\alpha f\\
w^{1+\alpha}J^{-1/\alpha}\partial_tH=w^{1+\alpha}J^{-1/\alpha}g \\
G(t=0)=\partial_tG(t=0)=H(t=0)=0
\end{cases}
\end{equation} 
and the solution satisfies 
\begin{equation}\label{reg1}
\|(\partial_tG,G,H)\|_{C([0,T];X^{\alpha,0}\times Y^{\alpha,0}\times Z^{\alpha,0})}\leq C\|(f,g)\|_{L^1(X^{\alpha,0}\times Z^{\alpha,0})}
\end{equation}
Moreover, if  $f\in X^{\alpha, N-2}$, $g\in Z^{\alpha,N-2}$, then 
\begin{equation}\label{reg2}
\|(\partial_tG,G,H)\|_{C([0,T];X^{\alpha,N-2}\times Y^{\alpha,N-2}\times Z^{\alpha,N-2})}\leq C\|(f,g)\|_{L^1(X^{\alpha,N-2}\times Z^{\alpha,N-2})}
\end{equation}
for some constant $C $ that depends only on $\mathcal{TE}^N(\eta,\partial_t\eta)$. 
\end{proposition}

Let $\mathcal{A}$ denote the set 
\[
\mathcal{A}=\left\{
\left(\begin{array}{c} 
\phi\\
\psi 
\end{array}
\right) \in C^\infty((0,\infty)\times \Omega)\; \text{ such that } \;
(\phi,\partial_t\phi,\psi)_{t=T}=0  \right\}
\]
Hence, $(\partial_tG,G,H)$ solves \eqref{linear} on a time interval $(0,T)$ if and only if for each test function $(\phi,\psi)\in\mathcal{A}$,  we have 
\begin{equation}
\begin{split}\label{weakf}
\int_0^T\int  G \cdot ( w^\alpha\partial_t^2 \phi +\mathcal{L}^e{\phi}+\mathcal{L}^d{\phi})  
-\frac12 w^{1+\alpha}J^{-1/\alpha} H:\text{Curl}_\eta {\phi}  \,dx dt = \int_0^T\int  w^\alpha f \phi\, dx dt\\
\int_0^T\int 4 w^{1+\alpha}J^{-1/\alpha} H : ( \partial_t \psi -\partial_t(J^{-1/\alpha}) \psi)   dxdt =\int_0^T\int -4 w^{1+\alpha}J^{-1/\alpha} g \psi dxdt
\end{split}
\end{equation}
We denote 
\[
\mathcal{V}\binom{G}{H}=\binom{w^\alpha\partial_t^2G+\mathcal{L}^eG+\mathcal{L}^dG+\mathcal{L}^cH}
{-w^{1+\alpha}J^{-1/\alpha}\partial_tH}
\]
defined on the core 
\[
\left\{ \binom{G}{H} \Big| \partial_t^2G\in L_t^2X^{\alpha,0}, \partial_tH \in L_t^2Z^{\alpha,0}, G\in L_t^2(\mathcal{D}(\mathcal{L}^e)) \right\}
\]
Hence $\mathcal{V}$ can be  extended uniquely to a closed operator. Moreover, $\mathcal{A}\subset \mathcal{D}\mathcal{V}^\ast$, the dual of $\mathcal{V}$, and 
\[
\mathcal{V}^\ast\binom{\phi}{\psi}=\binom{w^\alpha\partial_t^2\phi+\mathcal{L}^e\phi+\mathcal{L}^d\phi}
{ w^{1+\alpha}J^{-1/\alpha}[4\partial_t\psi -4\partial_t(J^{-1/\alpha})\psi-\frac{1}{2}\text{Curl}_\eta\phi ]}
\]
Therefore, \eqref{weakf} holds for each $(\phi,\psi)\in \mathcal{A}$ if and only if for each $(\phi,\psi)\in \mathcal{A}$, we have 
\begin{equation}\label{equif}
\int_0^T\int \binom{G}{H}\cdot \mathcal{V}^\ast \binom{\phi}{\psi}=\int_0^T\int 
\binom{w^\alpha f}{-4w^{1+\alpha}J^{-1/\alpha}g}\cdot\binom{\phi}{\psi}
\end{equation}

We take $\binom{\phi}{\psi}\in A$ and denote 
$$\mathcal{V}^\ast\binom{\phi}{\psi}=\binom{\Phi}{\Psi}$$ The energy estimates 
for $\mathcal{V}^\ast$ with \eqref{AE} and \eqref{ee} yield that 
\[
\sup_{0\leq t\leq T}\{\frac12\|\partial_t\phi\|^2_{X^{\alpha,0}}+\frac16\|\phi\|^2_{Y^{\alpha,0}}+\frac12 
\|\psi\|^2_{Z^{\alpha,0}}\} \leq C\int_0^T\|\Phi\|_{{X^{\alpha,0}}^\ast}^2+\|\Psi\|_{{Z^{\alpha,0}}^\ast}^2 dt
\]
where ${X^{\alpha,0}}^\ast$ and ${Z^{\alpha,0}}^\ast$ denote the dual spaces of $X^{\alpha,0}$ and $Z^{\alpha,0}$. 
Thus the operator $\mathcal{V}^\ast$ defines a bijection between $\mathcal{A}$ and $\mathcal{V}^\ast(\mathcal{A})$. Let $S_0$ be its inverse. Hence 
\[
S_0:\mathcal{V}^\ast(\mathcal{A})\rightarrow A \text{ given by }
\binom{\Phi}{\Psi}|\rightarrow \binom{\phi}{\psi}
\]
and we have 
\[
\|(\partial_t\phi,\phi,\psi)\|_{C([0,T];X^{\alpha,0}\times Y^{\alpha,0}\times Z^{\alpha,0})}\leq 
C \|\binom{\Phi}{\Psi}\|_{L^1({X^{\alpha,0}}^\ast\times {Z^{\alpha,0}}^\ast)}
\]
We extend this operator by density to $\overline{\mathcal{V}^\ast(\mathcal{A})}^{L^1({X^{\alpha,0}}^\ast\times {Z^{\alpha,0}}^\ast)}$ and to $L^1 ({X^{\alpha,0}}^\ast\times {Z^{\alpha,0}}^\ast)$ by Hahn-Banach. We denote this extension by $S$: 
\[
S:L^1 ({X^{\alpha,0}}^\ast\times {Z^{\alpha,0}}^\ast) \rightarrow C([0,T];X^{\alpha,0}\times Y^{\alpha,0}\times Z^{\alpha,0}) \;,\quad \binom{\Phi}{\Psi}|\rightarrow \binom{\phi}{\psi}
\]

Now we want to solve \eqref{linear}, namely $\mathcal{V}\binom{G}{H}= \binom{w^\alpha f}{w^{1+\alpha}J^{-1/\alpha}g}$ with $G(t=0)=\partial_t G(t=0)=H(t=0)=0$. This is equivalent to requiring that  \eqref{equif} holds for each $(\phi,\psi)\in\mathcal{A}$. Hence, it is enough to show that for all 
$\binom{\Phi}{\Psi}\in L^1({{X^{\alpha,0}}^\ast\times {Z^{\alpha,0}}^\ast})$, we have 
\begin{equation*}
\int_0^T\int \binom{G}{H}\cdot \binom{ \Phi}{\Psi}=\int_0^T\int 
\binom{w^\alpha f}{-4w^{1+\alpha}J^{-1/\alpha}g}\cdot S\binom{\Phi}{\Psi}
\end{equation*}
Therefore, it is enough to take $\binom{G}{H}=S^\ast \binom{w^\alpha f}{-4w^{1+\alpha}J^{-1/\alpha}g}$ where $S^\ast$ is the dual of $S$, which satisfies 
\[
S^\ast : \mathcal{M}(0,T;{X^{\alpha,0}}^\ast\times {Z^{\alpha,0}}^\ast)\rightarrow L^\infty(0,T;X^{\alpha,0}\times Y^{\alpha,0}\times Z^{\alpha,0})
\]
and thus $\|(\partial_tG,G,H)\|_{L^\infty(0,T;X^{\alpha,0}\times Y^{\alpha,0}\times Z^{\alpha,0})}
\leq C\|(f,g)\|_{L^1(X^{\alpha,0}\times Z^{\alpha,0})}$. At this stage we do not know whether 
$(G,H)$ is continuous in \eqref{reg1}. This actually follows from the regularity and the density argument. 

The uniqueness of $(G,H)$ will follows from the fact that if $f=g=0$ in \eqref{linear}, then $(0,0)$ is the only solution to \eqref{linear} in $L^{\infty}((0,T;X^{\alpha,0}\times Y^{\alpha,0}\times Z^{\alpha,0})$. To prove it, consider a solution $(\partial_tG,G,H)\in L^{\infty}((0,T;X^{\alpha,0}\times Y^{\alpha,0}\times Z^{\alpha,0})$ to \eqref{linear} with $f=g=0$. We will also make use of the duality argument. Indeed, as changing the roles of $\mathcal{V}$ and $\mathcal{V}^\ast$ and arguing as above, we can prove the existence of a solution $(\phi,\psi)$ to the dual problem 
\begin{equation}\label{dualp}
\begin{cases}
w^\alpha\partial_t^2\phi+\mathcal{L}^e\phi+\mathcal{L}^d\phi = \Phi\\
w^{1+\alpha}J^{-1/\alpha}[4\partial_t\psi -4\partial_t(J^{-1/\alpha})\psi 
-\frac12\text{Curl}_\eta\phi]=\Psi \\
\phi(t=0)=\partial_t\phi(t=0)=\psi(t=0)=0
\end{cases}
\end{equation}
for each $(\Phi,\Psi)\in L^1({X^{\alpha,0}}^\ast\times {Z^{\alpha,0}}^\ast)$. Then for each $(\Phi,\Psi)\in L^1({X^{\alpha,0}}^\ast\times {Z^{\alpha,0}}^\ast)$, we consider $(\phi,\psi)$ a solution to \eqref{dualp}. Hence, we 
can write \eqref{equif} with the solution $(\phi,\psi)$. This yields 
\[
\int_0^T \int \binom{G}{H}\cdot \binom{\Phi}{\Psi} dxdt =0
\] 
which implies $G=H=0$. 

The proof of \eqref{reg2} is an easy modification of the above argument based on induction, duality and density arguments as done in \cite{JM}. The tangential regularity can be obtained in the same way. The only difference for the normal regularity is that the degeneracy of the linear operators $\mathcal{L}^e,\mathcal{L}^d$ in \eqref{linear} will change  according to the number of $\partial_3$ for $G$ and thus the appropriate function space is $X^{\alpha,k}$ as in the same spirit of  Section \ref{proof-er}.  We omit the details.

\subsection{Formal derivation of the equation for $G$ and $H=\text{Curl}_\eta G$}\label{GH}
 
 We first derive the equation for $F\equiv w^{-\alpha}\partial_3(w^{1+\alpha}\partial_3\eta)=w\partial_3^2\eta +(1+\alpha)\partial_3w\,\partial_3\eta$. From \eqref{3n}, we obtain 
\[
\begin{split}
& \partial_t^2F^i+\underline{\frac{1}{w^{1+\alpha}}\left(w^{2+\alpha}
\underline{\{w\partial_3^2[A_i^kJ^{-1/\alpha}]
+(1+\alpha)\partial_3w\,\partial_3[A_i^kJ^{-1/\alpha}]\}}_{(\star\star)}\right),_k}_{(\star)}\\
& - (1+\alpha)w \,\partial_3 w,_k \partial_3[A_i^kJ^{-1/\alpha}] + w \,I^{0,2} +(1+\alpha)\partial_3w\, I^{0,1}=0
\end{split}
\]
First we look at $(\star\star)$
\[
\begin{split}
 &w\partial_3^2[A_i^kJ^{-1/\alpha}]
+(1+\alpha)\partial_3w\,\partial_3[A_i^kJ^{-1/\alpha}] \\
&= -\{J^{-1/\alpha}A^k_rA^s_i+\tfrac1\alpha J^{-1/\alpha}A^k_iA^s_r\}\{w\partial_3^2\eta^r,_s+(1+\alpha)\partial_3w\,
\partial_3\eta^r,_s\} \\
&\quad
-w\partial_3\{J^{-1/\alpha}A^k_rA^s_i+\tfrac1\alpha J^{-1/\alpha}A^k_iA^s_r\}\cdot \partial_3\eta^r,_s\\
&=-\{J^{-1/\alpha}A^k_rA^s_i+\tfrac1\alpha J^{-1/\alpha}A^k_iA^s_r\} F^r,_s\\
&\quad+ 
\{J^{-1/\alpha}A^k_rA^s_i+\tfrac1\alpha J^{-1/\alpha}A^k_iA^s_r\}\{w,_s\partial_3^2\eta^r
+(1+\alpha)\partial_3w,_s\partial_3\eta^r\}\\
&\quad
-w\partial_3\{J^{-1/\alpha}A^k_rA^s_i+\tfrac1\alpha J^{-1/\alpha}A^k_iA^s_r\}\cdot \partial_3\eta^r,_s
\end{split}
\]
Now we rewrite $(\star)$ as follows: 
\[
 \begin{split}
  (\star)&=\frac{1}{w^{1+\alpha}}\left(w \cdot w^{1+\alpha}\{w\partial_3^2[A_i^kJ^{-1/\alpha}]
+(1+\alpha)\partial_3w\,\partial_3[A_i^kJ^{-1/\alpha}]\}\right),_k\\
&=\frac{1}{w^{\alpha}}\left( w^{1+\alpha}\{w\partial_3^2[A_i^kJ^{-1/\alpha}]
+(1+\alpha)\partial_3w\,\partial_3[A_i^kJ^{-1/\alpha}]\}\right),_k \\
&\quad + w,_k \{w\partial_3^2[A_i^kJ^{-1/\alpha}] 
+(1+\alpha)\partial_3w\,\partial_3[A_i^kJ^{-1/\alpha}]\}\\
&=-\frac{1}{w^{\alpha}}\left( w^{1+\alpha}\{J^{-1/\alpha}A^k_rA^s_i+\tfrac1\alpha J^{-1/\alpha}A^k_iA^s_r\} F^r,_s\right),_k \\
&\quad+\underline{\frac{1}{w^{\alpha}}\left( w^{1+\alpha}\{J^{-1/\alpha}A^k_rA^s_i+\tfrac1\alpha J^{-1/\alpha}A^k_iA^s_r\}\{w,_s\partial_3^2\eta^r+(1+\alpha)\partial_3w,_s\partial_3\eta^r\}\right),_k}_{(\ast)}\\
&\quad +\underline{w,_k \{w\partial_3^2[A_i^kJ^{-1/\alpha}] 
+(1+\alpha)\partial_3w\,\partial_3[A_i^kJ^{-1/\alpha}]\}}_{(\ast\ast)}\\
&\quad-\frac{1}{w^{\alpha}}( w^{2+\alpha}\partial_3\{J^{-1/\alpha}A^k_rA^s_i+\tfrac1\alpha J^{-1/\alpha}A^k_iA^s_r\}\cdot \partial_3\eta^r,_s),_k \\
 \end{split}
\]
We now show that the undesirable terms $w\partial_3^3\eta$ and 
$\partial_3^2\eta$ cancel out in $(\ast)+(\ast\ast)$ and thus the dominant terms are of $w\partial_3^2\eta,_\sigma$ and 
$\partial_3\eta,_\sigma$. Indeed, $(\ast)+(\ast\ast)$ can be written as follows:
\[
 \begin{split}
  &(\ast)+(\ast\ast)\\
&={w^{-\alpha}}\left( w^{1+\alpha}\{J^{-1/\alpha}A^\kappa_rA^s_i+\tfrac1\alpha J^{-1/\alpha}A^\kappa_iA^s_r\}\{w,_s\partial_3^2\eta^r+(1+\alpha)\partial_3w,_s\partial_3\eta^r\}\right),_\kappa\\
&+ wJ^{-1/\alpha} \{ [A^3_rA^\sigma_i+\tfrac{1}{\alpha}A^3_i A^\sigma_r]w,_\sigma\partial_3^3\eta^r- [A^\kappa_rA^s_i+\tfrac{1}{\alpha}A^\kappa_i A^s_r]w,_\kappa\partial_3^2\eta^r,_s- [A^k_rA^\sigma_i+\tfrac{1}{\alpha}A^k_i A^\sigma_r]\\
&\quad\;\cdot w,_k\partial_3^2\eta^r,_\sigma+  [A^3_rA^s_i+\tfrac{1}{\alpha}A^3_i A^s_r] [(2+\alpha) \partial_3w,_s\partial_3^2\eta^r+(1+\alpha)\partial_3^2w,_s\partial_3\eta^r]\}\\
& +(1+\alpha)w,_3J^{-1/\alpha}\{[A^3_rA^\sigma_i+\tfrac{1}{\alpha}A^3_i A^\sigma_r]w,_\sigma\partial_3^2\eta^r- [A^\kappa_rA^s_i+\tfrac{1}{\alpha}A^\kappa_i A^s_r]w,_\kappa\partial_3\eta^r,_s\\
&\quad- [A^k_rA^\sigma_i+\tfrac{1}{\alpha}A^k_i A^\sigma_r]w,_k\partial_3\eta^r,_\sigma 
+ (1+\alpha) [A^3_rA^s_i+\tfrac{1}{\alpha}A^3_i A^s_r]  \partial_3w,_s\partial_3\eta^r \} \\
&+ w \{ \partial_3[J^{-1/\alpha}A^3_rA^\sigma_i+\tfrac{1}{\alpha}J^{-1/\alpha}A^3_i A^\sigma_r]w,_\sigma\partial_3^2\eta^r-\partial_3 [J^{-1/\alpha}A^\kappa_rA^s_i+\tfrac{1}{\alpha}J^{-1/\alpha}A^\kappa_i A^s_r]w,_\kappa\partial_3\eta^r,_s\\
&-\partial_3[J^{-\frac{1}{\alpha}}A^k_rA^\sigma_i+\tfrac{1}{\alpha}J^{-\frac{1}{\alpha}}A^k_i A^\sigma_r]w,_k\partial_3\eta^r,_\sigma +(1+\alpha)\partial_3 [J^{-\frac{1}{\alpha}}A^3_rA^s_i+\tfrac{1}{\alpha}J^{-\frac{1}
{\alpha}}A^3_i A^s_r] \partial_3w,_s\partial_3\eta^r \}
\end{split}
\]
Therefore, the equation for $F$ reads as follows. 
\begin{equation}\label{F}
 \begin{split}
 & \partial_t^2 F^i -w^{-\alpha} ( w^{1+\alpha}\{J^{-1/\alpha}A^k_rA^s_i+\tfrac1\alpha 
J^{-1/\alpha}A^k_iA^s_r\} F^r,_s),_k \\
&-{w^{-\alpha}}( w^{2+\alpha}\partial_3\{J^{-1/\alpha}A^k_rA^s_i+\tfrac1\alpha J^{-1/\alpha}A^k_iA^s_r\}\cdot \partial_3\eta^r,_s),_k\\
&+{w^{-\alpha}}\left( w^{1+\alpha}\{J^{-1/\alpha}A^\kappa_rA^s_i+\tfrac1\alpha J^{-1/\alpha}A^\kappa_iA^s_r\}\{w,_s\partial_3^2\eta^r+(1+\alpha)\partial_3w,_s\partial_3\eta^r\}\right),_\kappa\\
&+ wJ^{-1/\alpha} \{ [A^3_rA^\sigma_i+\tfrac{1}{\alpha}A^3_i A^\sigma_r]w,_\sigma\partial_3^3\eta^r- [A^\kappa_rA^s_i+\tfrac{1}{\alpha}A^\kappa_i A^s_r]w,_\kappa\partial_3^2\eta^r,_s- [A^k_rA^\sigma_i+\tfrac{1}{\alpha}A^k_i A^\sigma_r]\\
&\quad\;\cdot w,_k\partial_3^2\eta^r,_\sigma+  [A^3_rA^s_i+\tfrac{1}{\alpha}A^3_i A^s_r] [(2+\alpha) \partial_3w,_s\partial_3^2\eta^r+(1+\alpha)\partial_3^2w,_s\partial_3\eta^r]\}\\
& +(1+\alpha)w,_3J^{-1/\alpha}\{[A^3_rA^\sigma_i+\tfrac{1}{\alpha}A^3_i A^\sigma_r]w,_\sigma\partial_3^2\eta^r- [A^\kappa_rA^s_i+\tfrac{1}{\alpha}A^\kappa_i A^s_r]w,_\kappa\partial_3\eta^r,_s\\
&\quad- [A^k_rA^\sigma_i+\tfrac{1}{\alpha}A^k_i A^\sigma_r]w,_k\partial_3\eta^r,_\sigma 
+ (1+\alpha) [A^3_rA^s_i+\tfrac{1}{\alpha}A^3_i A^s_r]  \partial_3w,_s\partial_3\eta^r \} \\
&+ w \{ \partial_3[J^{-1/\alpha}A^3_rA^\sigma_i+\tfrac{1}{\alpha}J^{-1/\alpha}A^3_i A^\sigma_r]w,_\sigma\partial_3^2\eta^r-\partial_3 [J^{-1/\alpha}A^\kappa_rA^s_i+\tfrac{1}{\alpha}J^{-1/\alpha}A^\kappa_i A^s_r]\\
&\quad\;\cdot w,_\kappa\partial_3\eta^r,_s
-\partial_3[J^{-\frac{1}{\alpha}}A^k_rA^\sigma_i+\tfrac{1}{\alpha}J^{-\frac{1}{\alpha}}A^k_i A^\sigma_r]w,_k\partial_3\eta^r,_\sigma \\
&\quad \;+(1+\alpha)\partial_3 [J^{-\frac{1}{\alpha}}A^3_rA^s_i+\tfrac{1}{\alpha}J^{-\frac{1}
{\alpha}}A^3_i A^s_r] \partial_3w,_s\partial_3\eta^r \}\\
&- (1+\alpha)w \,\partial_3 w,_k \partial_3[A_i^kJ^{-1/\alpha}] + w \,I^{0,2} +(1+\alpha)\partial_3w\, I^{0,1}=0
 \end{split}
\end{equation}
The first line is the main part, the second line has full derivative but with the desirable weight $w^2$,  the rest of lines are either of  lower order with respect to $\partial_3$ with appropriate weights. We denote the second line through the last line by $R^i$. Next by using the fact $A^k_i=\delta^k_sA^s_i=A^k_r\eta^r,_sA^s_i$, we write the equation \eqref{deacoustic} as follows:  
\begin{equation}\label{etatt}
\begin{split}
 \eta_{tt}^i -w^{-\alpha} (w^{1+\alpha} J^{-1/\alpha}A^k_rA^s_i\eta^r,_s),_k-\tfrac{1}{\alpha}
 w^{-\alpha}  (w^{1+\alpha} J^{-1/\alpha}A^k_iA^s_r\eta^r,_s),_k &\\
 +(2+\tfrac{3}{\alpha})w^{-\alpha}(w^{1+\alpha}J^{-1/\alpha} A^k_i),_k&=0
\end{split}
\end{equation}
Now combining  \eqref{F} and \eqref{etatt} with the equation for $[\partial_1^2+\partial_2^2]\eta$ in 
\eqref{3nm}, the equation for $G= [-\partial_1^2-\partial_2^2-w^{-\alpha}\partial_3w^{1+\alpha}\partial_3+\lambda]\eta$ reads as follows: 
\begin{equation}
\begin{split}\label{G}
&\partial_t^2 G^i - w^{-\alpha} ( w^{1+\alpha}\{J^{-1/\alpha}A^k_rA^s_i+\tfrac1\alpha 
J^{-1/\alpha}A^k_iA^s_r\} G^r,_s),_k \\
&=-w^{-\alpha} (w^{1+\alpha}\partial_\beta[{J}^{-1/\alpha}{A}^k_r{A}^s_i+\tfrac1\alpha {J}^{-1/\alpha}{A}^k_i{A}^s_r]\cdot \partial_\beta{\eta}^r,_s),_k + I^{2,0,0}\\
&\quad+R^i -\lambda (2+\tfrac{3}{\alpha})w^{-\alpha}(w^{1+\alpha}J^{-1/\alpha} A^k_i),_k  \equiv \mathcal{R}^i
\end{split}
\end{equation}
We note that $\mathcal{R}^i$ consists of essentially lower-order terms and in particular, 
\[
\|\mathcal{R}^i\|_{X^{\alpha,N-2}}<\infty \;\text{ if }\;\mathcal{E}^N(\eta,\partial_t\eta)<\infty\text{ and }\|Dw\|_{X^{\alpha,N}}<\infty \,.
\]
The equation for $H=\text{Curl}_\eta G$ can be derived in the same way from the curl equation  \eqref{curleta0}. 
\begin{equation}
\begin{split}\label{H}
&\partial_t H =[-\partial_\beta^2-w\partial_3^2-(2+\alpha)\partial_3w\,\partial_3+\lambda][\text{Curl} u_0]^k_j\\
&-\partial_t{A}^s_j[\partial_\beta^2+w\partial_3^2+(2+\alpha)\partial_3w\partial_3+\lambda]
{\eta}^k,_s-\partial_t{A}^s_k[\partial_\beta^2+w\partial_3^2+(2+\alpha)\partial_3w\partial_3+\lambda]
{\eta}^j,_s\\
&+\sum_{p=1}^2(\partial_\beta^p{A}^s_j
\cdot\partial_t\partial_\beta^{2-p}{\eta}^k,_s+w\partial_3^p{A}^s_j\cdot \partial_t
\partial_3^{2-p}{\eta}^k,_s\\
&\quad\quad\;\;- \partial_\beta^p{A}^s_k
\cdot\partial_t\partial_\beta^{2-p}{\eta}^j,_s-w\partial_3^p{A}^s_k\cdot \partial_t
\partial_3^{2-p}{\eta}^j,_s) \\
&+(2+\alpha)\partial_3w(\partial_3{A}^s_j\cdot\partial_t 
{\eta}^k,_s-\partial_3{A}^s_k\cdot\partial_t 
{\eta}^j,_s) +\lambda\int_0^t (\partial_tA^s_j\cdot\partial_t\eta^k,_s- \partial_tA^s_k\cdot \partial_t\eta^j,_s) d\tau \\
&-\sum_{p=0}^2\int_0^t(\partial_t\partial_\beta^p{A}^s_j
\cdot\partial_t\partial_\beta^{2-p}{\eta}^k,_s+w\partial_t\partial_3^p{A}^s_j\cdot \partial_t
\partial_3^{2-p}{\eta}^k,_s\\
&\quad\quad\quad\;\;\; -\partial_t\partial_\beta^p{A}^s_k
\cdot\partial_t\partial_\beta^{2-p}{\eta}^j,_s-w\partial_t\partial_3^p{A}^s_k\cdot \partial_t
\partial_3^{2-p}{\eta}^j,_s)d\tau\\
&-\sum_{p=0}^1
 (2+\alpha)\partial_3w\int_0^t (\partial_t\partial_3^p{ A}^s_j\cdot
 \partial_t\partial_3^{1-p}{\eta}^k,_s -\partial_t\partial_3^p{ A}^s_k\cdot
 \partial_t\partial_3^{1-p}{\eta}^j,_s )d\tau \\
 &+\partial_3 w\partial_t ({A}^\sigma_j\cdot\partial_3{\eta}^k,_\sigma-{A}^\sigma_k\cdot\partial_3{\eta}^j,_\sigma) -\partial_t( w,_\sigma A^\sigma_j\cdot \partial_3^2\eta^k- w,_\sigma A^\sigma_k\cdot \partial_3^2\eta^j) \\
 &- (2+\alpha)\partial_t(\partial_3w,_s {A}^s_j\cdot\partial_3{\eta}^k-
\partial_3w,_s {A}^s_k\cdot\partial_3{\eta}^j)\equiv [\mathcal{S}]^k_j
\end{split}
\end{equation}
We note that 
\[
\|\mathcal{S}\|_{Z^{\alpha,N-2}}<\infty \;\text{ if }\;\mathcal{E}^N(\eta,\partial_t\eta)<\infty \text{ and }\|Dw\|_{X^{\alpha,N}}<\infty \,.
\]

\section{General smooth initial domain $\Omega$}\label{genera-domain}

Here we would like to discuss the changes to be made to the argument to prove 
theorem \ref{thm-general} for general domains. 

There are different ways of trying to extend the result to the general case. One can 
use $K$  charts   to cover the boundary of $\Omega$ and one chart for the interior. 
Then, one can   use   change of coordinates to straighten out the boundary
for each chart.   One can then prove a priori estimates. However, we think 
 that one  of the disadvantage of this method  is that the proof of existence 
of approximate solutions is technical since one has to solve  $K+1$ problems 
simultaneously at each step. Here, we will present a more geometric  way 
motivated by Shatah and Zeng \cite{SZ08}.  

Recall the notations introduced in Section \ref{energy-main}.  
The main difference between the tangential and normal derivatives is the fact that 
$ |\partial^m_\beta  w |  \leq C  w  $ and that $  \frac1C \leq \partial_\zeta   w  \leq C.    $
 
The proof of the a priori estimates, namely Proposition \ref{prop} is identical. 
One has just to replace $\partial_3$ by $\partial_\zeta$ and $\partial_1, \partial_2$ 
by $\partial_\beta$ for $\beta \in \T$. 

When solving the iteration scheme, we replace the definition of $G$ by 
\begin{equation} \label{new-G}
G=    [ \sum_{\beta \in \T}  (\partial_\beta)^* \partial_\beta  - 
  w^{-\alpha} \partial_\zeta w^{1+ \alpha} \partial_\zeta 
     +\lambda ]   u  = \G  u    
\end{equation}
we recall that if $\partial_\beta = a^\beta_1(x) \partial_1 +a^\beta_2(x) \partial_2 + a^\beta_3(x) \partial_3 $
then $ (\partial_\beta)^*(.) =   \partial_1 ( a^\beta_1(x) . )    +   
  \partial_2 (  a^\beta_2(x)  . )  +   \partial_3  ( a^\beta_3(x) . ) . $
 Actually, a better choice of the operator  $\G$ is to take it to be selfadjoint in $L^2(w^\alpha dx)$ and 
hence  take 
\begin{equation} \label{new-G-self}
G=    [ \sum_{\beta \in \T}  w^{-\alpha}  (\partial_\beta)^*  w^\alpha   \partial_\beta  - 
  w^{-\alpha} (\partial_\zeta)^*  w^{1+ \alpha} \partial_\zeta 
     +\lambda ]   u  = \G  u.     
\end{equation}
As in section \ref{IS}, we solve the following approximate system for 
$G_{\nu+1}$ and $H_{\nu+1}$: 
\begin{equation}\label{approx1}
 \begin{split}
 w^\alpha\partial_t^2{G_{\nu+1}}
+\mathcal{L}^e_\nu\,{G_{\nu+1}}
+\mathcal{L}^d_\nu\,{G_{\nu+1}}
+ \mathcal{L}^c_\nu\,{H_{\nu+1}} &= w^\alpha\mathcal{R}_\nu \\
\partial_t H_{\nu+1} &= \mathcal{S}_\nu
 \end{split}
\end{equation}
where $\mathcal{R}_\nu$ and $\mathcal{S}_\nu$ are  given by by formulae that are similar to 
those in section \ref{IS} and consist of lower order terms. The rest of the proof is identical. 
We just mention that Lemma \ref{er} was proved above in the flat case. Extending it to the 
case of the operator $\G$ defined in \eqref{new-G} 
 or \eqref{new-G-self} can be easily done following the arguments in \cite{Baouendi67,BG69}.

\section{Discussion}\label{discussion}

The study of vacuum states in gas and fluid dynamics can be tracked back at least to a conference, Problems of Cosmical Aerodynamics, held in Paris,1949 where one session chaired by J. von Neumann was devoted to the existence and uniqueness or multiplicity of solutions of the aerodynamical equations \cite{vonN}. See also the review article by Serre \cite{S2}. The last question raised by von Neumann concerns the validity of the Euler system in the presence of a vacuum:\\ 

\textit{There is a further difficulty in the expansion case considered by Burgers. It was accepted that the front advances into a vacuum. It is evident that you cannot get the normal conditions of kinetic theory here either, because the density of the gas goes to zero at the front, which means that the mean free path of the molecules will go to infinity. This means that if we are in the expanding gas and approach the (theoretical) front, we will necessarily come to regions where the mean free path is larger than the distance from the front. In such regions one cannot use the hydrodynamical equations. But, as in the case of the shock wave, where ordinary conditions are reached at a distance of a few mean free paths from the shock itself, so in the case of expansion into a vacuum, at a short distance from the theoretical front, one comes into regions where the mean free path is considerably smaller than the distance from the front, and where again the classical hydrodynamical equations can be applied. If this is applied to expanding interstellar clouds, I think that in order that the classical theory be true down to 1/1000 of the density of the clouds, it is necessary that the distance from the theoretical front should be of the order of a percent of a parsec.}\\

 This statement led to interesting discussions, which are still enlightening at present, among von Neumann and several participants including Heisenberg. An important comment made by Heisenberg was the role of boundary conditions related to the boundary layer theory. And von Neumann answered: \\
 
\textit{The boundary layer theory for a fluid of low viscosity certainly furnishes a monumental warning. The naive and yet prima facie seemingly reasonable procedure would be to apply the ordinary equations of the ideal fluid and then to expect that viscosity will somehow take care of itself in a narrow region along the wall. We have learned that this procedure may lead to great errors; a complete theory of the boundary layer may give you completely different conditions also for the flow in the bulk of the field. It is possible that the same discipline will be necessary for the boundary with a vacuum.}\\
 
  It is amazing to learn that the difficulty and importance of the problem with vacuum was already discussed more than 60 years ago.  Although the mathematical theory of vacuum and related subjects in the domain of  kinetic theory and boundary layer theory is far from being complete even now, it is very exciting to make  a first step towards a better understanding of the problem. 
  
 Our methodology developed in this paper turns out to be robust and it would shed some light on other vacuum free boundary problems in more general framework such as problems as discussed in Section \ref{other}. We believe that our energy method can be adapted to the relativistic Euler equations in a physical vacuum \cite{JLM}.  And in the upcoming paper \cite{JM2}, with this new approach, we will give different proofs of the well-posedness of compressible liquid in vacuum studied in  \cite{L1} and the well-posedness of smoother vacuum states of compressible Euler flow considered in \cite{LY1}.


\begin{thebibliography}{10}

\bibitem{IMA} Video of Discussion: ``Free boundary problems related to water waves,'' Summer program:
Nonlinear Conservation Laws and Applications, July 13-21, 2009, Institute for Mathematics and its Applications,
http://www.ima.umn.edu/videos/?id=915

\bibitem{ASL08}
\textsc{B.~Alvarez-Samaniego, D.~Lannes}:
\newblock Large time existence for 3{D} water-waves and asymptotics.
\newblock {\em Invent. Math.}, 171(3):485--541, 2008.

\bibitem{AM05}
\textsc{D.~Ambrose, N.~Masmoudi}: 
\newblock The zero surface tension limit of two-dimensional water waves.
\newblock {\em Comm. Pure Appl. Math.}, 58(10):1287--1315, 2005.

\bibitem{AM} \textsc{D.~Ambrose, N.~Masmoudi}: Well-posedness of 3D
vortex sheets with surface tension,  \textit{Commun. Math. Sci.}
\textbf{5}, 391-430 (2007)

\bibitem{Baouendi67}
\textsc{M.~S. Baouendi}: 
\newblock Sur une classe d'op\'erateurs elliptiques d\'eg\'en\'er\'es.
\newblock {\em Bull. Soc. Math. France}, 95:45--87, 1967.


\bibitem{BG69}
\textsc{M.~S. Baouendi, C.~Goulaouic}:
\newblock R\'egularit\'e et th\'eorie spectrale pour une classe d'op\'erateurs
  elliptiques d\'eg\'en\'er\'es.
\newblock {\em Arch. Rational Mech. Anal.}, 34:361--379, 1969.

\bibitem{Bouchut04}
\textsc{F.~Bouchut}:
\newblock {\em Nonlinear stability of finite volume methods for hyperbolic
  conservation laws and well-balanced schemes for sources}.
\newblock Frontiers in Mathematics. Birkh\"auser Verlag, Basel, 2004.


\bibitem{CF} \textsc{L.~Caffarelli, A.~Friedman}:
Regularity of the free boundary for the one-dimensional flow of gas in a porous medium.
\textit{Amer. J. Math.} \textbf{101} (1979), no. 6, 1193--1218.

\bibitem{Chemin90}
\textsc{J.-Y. Chemin}:
\newblock Dynamique des gaz \`a masse totale finie.
\newblock {\em Asymptotic Anal.}, \textbf{3}, 215-220 (1990)

\bibitem{Chen97}
\textsc{G.-Q. Chen}:
\newblock Remarks on {R}. {J}. {D}i{P}erna's paper: ``{C}onvergence of the
  viscosity method for isentropic gas dynamics'' [{C}omm.\ {M}ath.\ {P}hys.\
  {\bf 91} (1983), no.\ 1, 1--30; {MR}0719807 (85i:35118)].
\newblock {\em Proc. Amer. Math. Soc.}, 125(10):2981--2986, 1997.

\bibitem{CS07}
\textsc{D.~Coutand,  S.~Shkoller}:
\newblock Well-posedness of the free-surface incompressible {E}uler equations
  with or without surface tension.
\newblock {\em J. Amer. Math. Soc.}, 20(3):829--930 (electronic), 2007.

\bibitem{CS09}
\textsc{D.~Coutand, S.~Shkoller}:
\newblock Well-posedness in smooth function spaces for the moving-boundary 1-D compressible
Euler equations in physical vacuum
\newblock {\em preprint},  2009

\bibitem{CS10}
\textsc{D.~Coutand,  S.~Shkoller}:
\newblock Well-posedness in smooth function spaces for the moving-boundary 3-D compressible Euler equations in physical vacuum
\newblock {\em preprint},  2010

\bibitem{CSL09}
\textsc{D.~Coutand, S.~Shkoller, H.~Lindblad}: 
\newblock A priori estimates for the free-boundary 3D compressible Euler equations in physical vacuum \newblock{\em Commun. Math. Phys.}, \textbf{296},  (2010), 559--587

\bibitem{Daf} \textsc{C.~Dafermos}:
\newblock {\em Hyperbolic conservation laws in continuum physics}. Second edition. Grundlehren der Mathematischen Wissenschaften [Fundamental Principles of Mathematical Sciences], 325. \newblock Springer-Verlag, Berlin, 2005.

\bibitem{Diperna83}
\textsc{R.~J. DiPerna}:
\newblock Convergence of the viscosity method for isentropic gas dynamics.
\newblock {\em Comm. Math. Phys.}, 91(1):1--30, 1983.


\bibitem{dz} \textsc{B. Ducomet, A. Zlotnik}: Stabilization and
stability for the spherically symmetric Navier-Stokes-Poisson
system, \emph{Appl.Math.Lett.} \textbf{18} (2005), 1190--1198

\bibitem{Fei} \textsc{E.~Feireisl}:
\textit{ Dynamics of viscous compressible fluids. Oxford Lecture Series in Mathematics and its Applications}, 26.
{Oxford University Press},  2004

\bibitem{Fri} \textsc{K.~Friedrichs}:
Symmetric hyperbolic linear differential equations. \newblock {\em Comm. Pure Appl. Math.}, \textbf{7},
345–392 (1954)


\bibitem{GMS09} \textsc{P.~Germain, N.~Masmoudi, J.~Shatah}  {Global solutions for the gravity water waves equation in dimension 3}, 
preprint 2009


\bibitem{GHO08}
\textsc{L.~Giacomelli, H.~Kn{\"u}pfer,  F.~Otto}:
\newblock Smooth zero-contact-angle solutions to a thin-film equation around
  the steady state.
\newblock {\em J. Differential Equations}, 245(6):1454--1506, 2008.

\bibitem{Grad} \textsc{H.~Grad}:
Mathematical problems in magneto-fluid dynamics and plasma physics.  1963  \textit{Proc. Internat. Congr. Mathematicians (Stockholm, 1962)}  pp. 560--583 \textit{Inst. Mittag-Leffler, Djursholm}

\bibitem{Grassin98}
\textsc{M.~Grassin}:
\newblock Global smooth solutions to {E}uler equations for a perfect gas.
\newblock {\em Indiana Univ. Math. J.} \textbf{47}, 1397-1432 (1998)

\bibitem{GB} \textsc{H.P.~Greenspan, D.S.~Butler}:
On the expansion of a gas into vacuum,  \textit{J. Fluid Mech.}  \textbf{13} (1962) 101-119

\bibitem{HS} \textsc{D.~Hoff, D.~Serre}: The failure of continuous dependence on initial data for the Navier-Stokes equations of compressible flow. \textit{SIAM J. Appl. Math.}  \textbf{51}  (1991),  no. 4, 887--898

\bibitem{HMP} \textsc{F.~Huang, P.~Marcati, R.~Pan}:
Convergence to the Barenblatt solution for the compressible Euler
equations with damping and vacuum, \textit{Arch. Ration. Mech.
Anal.}  \textbf{176}, 1-24 (2005)




\bibitem{J0} \textsc{J.~Jang}: Nonlinear Instability in Gravitational Euler-Poisson system for $\gamma=6/5$,
\textit{ Arch. Ration. Mech. Anal.} \textbf{188}, 265-307 (2008)

\bibitem{J} \textsc{J.~Jang}:
Local wellposedness of viscous gaseous stars,
\textit{Arch. Ration. Mech.
Anal.} \textbf{195} (2010), no. 3, 797-863

\bibitem{JM} \textsc{J.~Jang, N.~Masmoudi}:
Well-posedness for compressible Euler equations with physical vacuum singularity,
 \textit{Comm. Pure Appl. Math.} \textbf{62}, 1327-1385 (2009)
 
 \bibitem{JM1} \textsc{J.~Jang, N.~Masmoudi}:
Vacuum in Gas and Fluid dynamics, accepted for publication in \textit{Proceedings of the IMA summer school on Nonlinear Conservation Laws and Applications}



 \bibitem{JM2} \textsc{J.~Jang, N.~Masmoudi}: Compressible fluids with vacuum boundary, 
 \textit{in preparation}
 
  \bibitem{JLM} \textsc{J.~Jang, P.~LeFloch, N.~Masmoudi}:
Well-posedness theory for fluids in a physical vacuum, 
 \textit{in preparation}

\bibitem{Kato} \textsc{T.~Kato}:
The Cauchy problem for quasi-linear symmetric hyperbolic systems. \textit{Arch. Rational
Mech. Anal.} \textbf{58}, 181-205 (1975)



\bibitem{KMP07}
\textsc{A.~Kufner, L.~Maligranda, L.-E. Persson}:
\newblock {\em The {H}ardy inequality}.
\newblock Vydavatelsk\'y Servis, Plze\v n, 2007.
\newblock About its history and some related results.


\bibitem{Lannes05} \textsc{D.~Lannes}:  
Well-posedness of the water-waves equations.  
\textit{J. Amer. Math. Soc.}  \textbf{18}  (2005),  no. 3, 605--654 

\bibitem{Lax} \textsc{P.~Lax}: Weak solutions of nonlinear hyperbolic equations
and their numerical computation,
\textit{Comm. Pure Appl. Math.} \textbf{7} (1954), 159-193.

\bibitem{LLX} \textsc{H-L. Li, J. Li, Z. Xin}: Vanishing of vacuum states and blow-up phenomena of the compressible Navier-Stokes equations.  \textit{Comm. Math. Phys.}  \textbf{281}  (2008),  no. 2, 401--444.


\bibitem{L1} \textsc{H.~Lindblad}: Well posedness for the motion of
a compressible liquid with free surface boundary, \textit{Comm.
Math. Phys.} \textbf{260}, 319-392 (2005)

\bibitem{Lind} \textsc{H.~Lindblad}: Personal communication, November 2009

\bibitem{LPS96}
P.-L. Lions, B.~Perthame, and P.~E. Souganidis.
\newblock Existence and stability of entropy solutions for the hyperbolic
  systems of isentropic gas dynamics in {E}ulerian and {L}agrangian
  coordinates.
\newblock {\em Comm. Pure Appl. Math.}, 49(6):599--638, 1996.



\bibitem{Lions} \textsc{P.-L.~Lions}: \textit{Mathematical topis in fluid mechanics}, Vol 2. Compressible Models.
 New York: Oxford University Press, 1998


\bibitem{L2} \textsc{T.-P.~Liu}:
Compressible flow with damping and vacuum, \textit{Japan J. Appl.
Math} \textbf{13}, 25-32 (1996)

\bibitem{LS} \textsc{T.-P. Liu, J. Smoller}:
On the vacuum state for isentropic gas dynamics equations,
\textit{Advances in Math.} \textbf{1}, 345-359 (1980)

\bibitem{LXZ} \textsc{T.-P. Liu, Z. Xin, T. Yang}: Vacuum states for compressible flow.
\textit{ Discrete Contin. Dynam. Systems}  \textbf{4}  (1998),  no. 1, 1--32

\bibitem{LY1} \textsc{T.-P. Liu, T. Yang}:
Compressible Euler equations with vacuum, \textit{J. Differential
Equations}  \textbf{140}, 223-237 (1997)

\bibitem{LY2} \textsc{T.-P. Liu, T. Yang}:
Compressible flow with vacuum and physical singularity,
\textit{Methods Appl. Anal.} \textbf{7}, 495-509 (2000)

\bibitem{LXY} \textsc{T. Luo, Z. Xin, T. Yang}:
Interface behavior of compressible Navier-Stokes equations with
vacuum, \textit{SIAM J. Math. Anal.} \textbf{31}, 1175-1191 (2000)

\bibitem{Majda84}
\textsc{A.~Majda}:
\newblock {\em Compressible fluid flow and systems of conservation laws in
  several space variables}, volume~53 of {\em Applied Mathematical Sciences}.
\newblock Springer-Verlag, New York (1984)

\bibitem{Makino92}
\textsc{T.~Makino}:
\newblock Blowing up solutions of the {E}uler-{P}oisson equation for the
  evolution of gaseous stars.
\newblock In {\em Proceedings of the Fourth International Workshop on
  Mathematical Aspects of Fluid and Plasma Dynamics (Kyoto, 1991)}, volume~21,
 615-624 (1992)

\bibitem{MU87}
\textsc{T.~Makino, S.~Ukai}:
\newblock Sur l'existence des solutions locales de l'\'equation
  d'{E}uler-{P}oisson pour l'\'evolution d'\'etoiles gazeuses.
\newblock {\em J. Math. Kyoto Univ.} \textbf{27}, 387-399 (1987)

\bibitem{MU95I} \textsc{T.~Makino, S.~Ukai}: \newblock Local smooth solutions of the relativistic Euler equation.
\newblock\textit{J. Math. Kyoto Univ.}  \textbf{35} (1995), no. 1, 105--114.

\bibitem{MUK86}
\textsc{T.~Makino, S.~Ukai,  S.~Kawashima}:
\newblock Sur la solution \`a support compact de l'\'equations d'{E}uler
  compressible.
\newblock {\em Japan J. Appl. Math.} \textbf{3}, 249-257 (1986)


\bibitem{Masmoudi08cpam}
\textsc{N.~Masmoudi}:  Well-posedness for the FENE dumbbell model of
polymeric flows.
  \textit{Comm. Pure Appl. Math.}  \textbf{61}  (2008),  no. 12, 1685--1714.
  
  
\bibitem{mom} \textsc{S. Matusu-Necasova, M. Okada, T. Makino}:
Free boundary problem for the equation of spherically symmetric
motion of viscous gas III, \emph{Japan J.Indust.Appl.Math.}
\textbf{14} (1997), 199-213

\bibitem{Nishida} \textsc{T.~Nishida}:
 Equations of fluid dynamics---free surface problems. Frontiers of the mathematical sciences:
1985 (New York,1985).  \textit{Comm. Pure Appl. Math.}  \textbf{39}  (1986),  no. S, suppl., S221--S238

\bibitem{S} \textsc{D. Serre}:
Solutions classiques globales des \'equations d'Euler pour un fluide
parfait compressible, \textit{ Ann. Inst. Fourier (Grenoble)}
\textbf{47}, 139-153 (1997)

\bibitem{S2} \textsc{D. Serre}:  Von Neumann's comments about existence and uniqueness for the initial-boundary value problem in gas dynamics, \textit{Bull. Amer. Math. Soc. (N.S.)} \textbf{47}  (2010), no. 1, 139--144.

\bibitem{SZ08}
\textsc{J.~Shatah, C.~Zeng}:
\newblock Geometry and a priori estimates for free boundary problems of the
  {E}uler equation.
\newblock {\em Comm. Pure Appl. Math.}, 61(5):698--744, 2008.

\bibitem{Sideris85}
\textsc{T.~C. Sideris}:
\newblock Formation of singularities in three-dimensional compressible fluids.
\newblock {\em Comm. Math. Phys.}, \textbf{101}, 475-485 (1985)

\bibitem{vonN} \textsc{Chairman: J.~von Neumann}: Discussion on the existence and uniqueness or multiplicity of solutions of the aerodynamical equations, Proceedings of the Symposium on the Motion of Gaseous Masses of Cosmical Dimensions held at Paris, August 16-19, 1949


\bibitem{Wu09} \textsc{S.~Wu}: {Almost global well-posedness of the 2-D full water wave problem}, accepted by Inventiones Mathematicae


\bibitem{Xin} \textsc{Z.~Xin}: Blowup of smooth solutions to the compressible Navier-Stokes equation
 with compact density.  \textit{Comm. Pure Appl. Math.}  \textbf{51}  (1998),  no. 3, 229--240

\bibitem{Y} \textsc{T. Yang}:
Singular behavior of vacuum states for compressible fluids,
\textit{Comput. Appl. Math.} \textbf{190}, 211-231 (2006)




\bibitem{ZZ08}
\textsc{P.~Zhang, Z.~Zhang}: 
\newblock On the free boundary problem of three-dimensional incompressible
  {E}uler equations.
\newblock {\em Comm. Pure Appl. Math.}, 61(7):877--940, 2008.

\bibitem{ZF2} \textsc{T. Zhang, D. Fang}: Global behavior of spherically symmetric Navier-Stokes-Poisson system with degenerate viscosity coefficients.  \textit{Arch. Ration. Mech. Anal.}  \textbf{191}  (2009),  no. 2, 195--243


\end{thebibliography}
\end{document}